\documentclass[11pt]{article}
\usepackage{graphicx} 

\usepackage{amsmath,amssymb,amsthm}

\usepackage{thmtools}
\usepackage{thm-restate}
\usepackage{lineno}

\usepackage{subfig}

\usepackage{tikz}
\usetikzlibrary{shapes.geometric}

\usepackage{enumerate}
\usepackage[shortlabels]{enumitem}

\usepackage{caption}

\usepackage{makecell}

\newtheorem{thm}{Theorem}[section]
\newtheorem{prob}[thm]{Problem}
\newtheorem{ques}[thm]{Question}

\newtheorem{lem}[thm]{Lemma}
\newtheorem{obs}[thm]{Observation}
\newtheorem{cla}{Claim}[section]

\theoremstyle{definition}
\newtheorem{defi}[thm]{Definition}

\newcommand{\sub}{\text{sub}}

\newcommand{\trif}{
   \begin{tikzpicture}
    \begin{scope}[every node/.style={circle, fill, draw,inner sep = 1.5pt}]
        \path (0,0) node(x)  {} 
          (1.5,0) node(y) {}
          (0.75,1) node(z) {};
    \end{scope}
    \begin{scope}[every node/.style={node font = \footnotesize}]
    \draw (node cs:name=x) -- (node cs:name=y) node[midway,below = 2pt] {1};
    \draw (node cs:name=y) --  (node cs:name=z)node[near start,above = 2pt] {1};
    \draw (node cs:name=z) --  (node cs:name=x)node[near end,above = 2pt] {1};
     \end{scope}
   \end{tikzpicture}
}
\newcommand{\tris}{
    \begin{tikzpicture}
        \begin{scope}[every node/.style={circle, fill, draw,inner sep = 1.5pt}]
            \path (0,0) node(x)  {} 
              (1.5,0) node(y) {}
              (0.75,1) node(z) {};
        \end{scope}
        \begin{scope}[every node/.style={node font = \footnotesize}]
        \draw (node cs:name=x) -- (node cs:name=y) node[midway,below = 2pt] {0};
        \draw (node cs:name=y) --  (node cs:name=z)node[near start,above = 2pt] {1};
        \draw (node cs:name=z) --  (node cs:name=x)node[near end,above = 2pt] {1};
         \end{scope}
       \end{tikzpicture}}
\newcommand{\trit}{
    \begin{tikzpicture}
        \begin{scope}[every node/.style={circle, fill, draw,inner sep = 1.5pt}]
            \path (0,0) node(x)  {} 
              (1.5,0) node(y) {}
              (0.75,1) node(z) {};
        \end{scope}
        \begin{scope}[every node/.style={node font = \footnotesize}]
        \draw (node cs:name=x) -- (node cs:name=y) node[midway,below = 2pt] {0};
        \draw (node cs:name=y) --  (node cs:name=z)node[near start,above = 2pt] {0};
        \draw (node cs:name=z) --  (node cs:name=x)node[near end,above = 2pt] {1};
         \end{scope}
       \end{tikzpicture}
}
\newcommand{\trifo}{\begin{tikzpicture}
    \begin{scope}[every node/.style={circle, fill, draw,inner sep = 1.5pt}]
        \path (0,0) node(x)  {} 
          (1.5,0) node(y) {}
          (0.75,1) node(z) {};
    \end{scope}
    \begin{scope}[every node/.style={node font = \footnotesize}]
    \draw (node cs:name=x) -- (node cs:name=y) node[midway,below = 2pt] {0};
    \draw (node cs:name=y) --  (node cs:name=z)node[near start,above = 2pt] {0};
    \draw (node cs:name=z) --  (node cs:name=x)node[near end,above = 2pt] {0};
     \end{scope}
   \end{tikzpicture}}

\newcommand\intertri[6]{
   \begin{tikzpicture}
    \begin{scope}[every node/.style={circle, fill, draw,inner sep = 1.5pt}]
        \path (0,0) node(a)  {} 
          (0,1) node(b) {}
          (0,2) node(c) {}
          (0,3) node(d) {}
          (2.5,1.5) node(e) {};
    \end{scope}
    \begin{scope}[every node/.style={node font = \footnotesize}]
    \draw (node cs:name=a) -- (node cs:name=b) node[midway, left = 2pt] {#1};
    \draw (node cs:name=c) --  (node cs:name=d)node[midway,left = 2pt] {#2};
    \draw (node cs:name=e) --  (node cs:name=a)node[midway,below] {#3};
    \draw (node cs:name=e) -- (node cs:name=b) node[midway,above] {#4};
    \draw (node cs:name=e) --  (node cs:name=c)node[midway,above,near end] {#5};
    \draw (node cs:name=e) --  (node cs:name=d)node[midway,above] {#6};
     \end{scope}
   \end{tikzpicture}
}

\newcommand\twotrif{
   \begin{tikzpicture}
    \begin{scope}[every node/.style={circle, fill, draw,inner sep = 1.5pt}]
        \path (0,0) node(a)  {} 
          (2,0) node(b) {}
          (1,1.5) node(c) {}
          (3,0) node(d) {}
          (5,0) node(e) {}
          (4,1.5) node(f) {};
    \end{scope}
    \begin{scope}[every node/.style={node font = \footnotesize}]
    \draw (node cs:name=a) node[below left] {$u_2$}-- (node cs:name=b)node[below = 1pt] {$u_3$} -- (node cs:name=c);
    \draw (node cs:name=c) node[above = 1pt] {$u_1$} -- (node cs:name=a);
    \draw (node cs:name=d) node[below = 1pt] {$u_5$} --  (node cs:name=e)node[below right] {$u_6$} -- (node cs:name=f)node[above = 1pt] {$u_4$};
    \draw (node cs:name=f) -- (node cs:name=d);
     \end{scope}
   \end{tikzpicture}
}
\newcommand\twotris{
   \begin{tikzpicture}
    \begin{scope}[every node/.style={circle, fill, draw,inner sep = 1.5pt}]
        \path (0,0) node(a)  {} 
          (2,0) node(b) {}
          (1,1.5) node(c) {}
          (3,0) node(d) {}
          (5,0) node(e) {}
          (4,1.5) node(f) {};
    \end{scope}
    \begin{scope}[every node/.style={node font = \footnotesize}]
    \draw (node cs:name=a) node[below left] {$u_2$} -- (node cs:name=b)node[below = 1pt] {$u_3$} -- (node cs:name=c);
    \draw (node cs:name=c) node[above = 1pt] {$u_1$} -- (node cs:name=a);
    \draw (node cs:name=d) node[below = 1pt] {$u_5$}-- (node cs:name=c);
    \draw (node cs:name=e) node[below right] {$u_6$} -- (node cs:name=c);
    \draw (node cs:name=d) --  (node cs:name=e) -- (node cs:name=f) node[above = 1pt] {$u_4$};
    \draw (node cs:name=f) -- (node cs:name=d);
    \draw (node cs:name=c) -- (node cs:name=f);
     \end{scope}
   \end{tikzpicture}
}
\newcommand\twotrit{
   \begin{tikzpicture}
    \begin{scope}[every node/.style={circle, fill, draw,inner sep = 1.5pt}]
        \path (0,0) node(a)  {} 
          (2,0) node(b) {}
          (1,1.5) node(c) {}
          (3,0) node(d) {}
          (5,0) node(e) {}
          (4,1.5) node(f) {};
    \end{scope}
    \begin{scope}[every node/.style={node font = \footnotesize}]
    \draw (node cs:name=a) node[below left] {$u_2$} -- (node cs:name=b)node[below = 1pt] {$u_3$} -- (node cs:name=c)node[above = 1pt] {$u_1$};
    \draw (node cs:name=c) -- (node cs:name=a);
    \draw (node cs:name=d) node[below = 1pt] {$u_5$} -- (node cs:name=c);
    \draw (node cs:name=e) node[below right] {$u_6$} -- (node cs:name=c);
    \draw (node cs:name=d) --  (node cs:name=e) -- (node cs:name=f);
    \draw (node cs:name=f) node[above = 1pt] {$u_4$}-- (node cs:name=d);
    \draw (node cs:name=f) -- (node cs:name=a);
    \draw (node cs:name=f) -- (node cs:name=b);
    \draw (node cs:name=c) -- (node cs:name=f);
     \end{scope}
   \end{tikzpicture}
}

\title{Divisible subdivisions of graphs in subdivisions of complete graphs} 
\author{Xinmin Hou$^{a,b}$\footnote{Email: xmhou@ustc.edu.cn (X. Hou), wangxiangyang@mail.ustc.edu.cn (X. Wang)},\quad Xiangyang Wang$^a$\\
\small $^{a}$ School of Mathematical Sciences\\
\small University of Science and Technology of China, Hefei, Anhui 230026, China.\\
\small$^b$ Hefei National Laboratory\\
\small University of Science and Technology of China, Hefei 230088, Anhui, China
}
\begin{document}

\maketitle

\begin{abstract}
Let $\mathbb{Z}_q$ denote the cyclic group of order $q$. A $\mathbb{Z}_q$-edge-weighted  $K_f$ is the complete graph $K_f$ equipped with  a weight function $\omega : E(K_f) \to \mathbb{Z}_q$. 
A subdivision of a graph $H$ in a $\mathbb{Z}_q$-edge-weighted $K_f$ is called a $q$-divisible subdivision of $H$ if every subdivision path has weight congruent to zero modulo $q$.
Let $q\ge 2$ be an integer and let $H$ be a graph with $n$ vertices and $m$ edges. Define $s_q(H)$ to be the smallest number $f$ such that every $\mathbb{Z}_q$-edge-weighted $K_{f}$ contains a $q$-divisible subdivision of $H$. 
Das, Draganić, and Steiner raised the following question (Problem 4.1 in [Tight bounds for divisible subdivisions, J. Combin. Theory, Ser. B 165 (2024) 1–19]): Given $q\in\mathbb{N}$ and a subcubic graph $H$ with $n$ vertices and $m$ edges, is it true $s_q(H)= m(q -  1) + n$? They also established the upper bound $s_q(H)\le 7mq+8n+14q$ for such a graph $H$. In this paper, we improve this bound by showing that $s_q(H)\le (2q -  1)m + 2n - 1 + 4q$, and establishing a sharper bound $s_p(H)\le \frac{3p - 1}{2}m - \frac{p - 1}{2}n + \frac{p + 1}{2}$ for prime $p$ and connected $H$. 
We resolve this problem in the case $q=2$ by proving that  $s_2(H) = m + n$ for any 5-degenerate graph $H$, and in the case $q\ge 2$ and $T$ being a tree, by showing that $s_q(T) = nq - q + 1$.
Let $s_q(H,t)$ be the minimum number $f$ such that every $\mathbb{Z}_q$-edge-weighted $K_f$ contains a $q$-divisible $t$-subdivision of $H$, where a $t$-subdivision of $H$ is a subdivision of $H$ such that each edge of $H$ is subdivided exactly $t$ times. 
We also prove that $s_2(H,1)= m + n$, where $H$ is a tree or a cycle  on $n$ vertices with $m$ edges.

\end{abstract}

\section{Introduction}


A {\em subdivision} of a graph $H$, denoted by $\text{sub}(H)$, is a graph obtained from $H$ by adding new vertices, called {\em subdivision vertices}, on its edges. In this process, each edge of $H$ is replaced by a path, which is referred to as a {\em subdivision path} of $H$.
A subdivision of $H$ is said to be {\em $q$-divisible} if the length of every subdivision path is divisible by $q$. The vertices of $\text{sub}(H)$ that correspond to the original vertices of $H$ are called {\em branch vertices}. Bollobás and Thomason \cite{bollobas} and  Komlós and Szemerédi \cite{komlos} showed that any graph with an average degree of at least $ct^2$, for some constant $c > 0$, contains a subdivision of $K_t$.  Alon, Krivelevich, and Sudakov \cite{AlonKS} established that for any fixed $\varepsilon > 0$, every $n$-vertex graph with average degree $\varepsilon n$ contains a 1-subdivision of $K_k$ with $k = \Omega(\sqrt{n})$.
Liu and Montgomery \cite{LiuM} confirmed a longstanding conjecture posed by Thomassen \cite{Thomassen1}, which asserted that for any integer $k\ge 1$, a sufficiently high average degree in a graph guarantees the existence of a balanced subdivision (a subdivision where every edge is subdivided the same number of times) of the complete graph $K_k$.
This was further improved in \cite{IJHOZ}, where it was shown that an average degree on the order of $\Theta(k^2)$ suffices, a bound that is tight up to a multiplicative constant.

Earlier results on finding subdivisions of a given graph with constrained subdivision path lengths in the literature only apply to graphs with sufficiently large average degree. It is therefore natural to investigate whether such specified subdivisions can also be found in sparse graphs. This leads to the following question:
\begin{ques}
Whether, in a minor or a subdivision of a complete graph, there exists a subdivision of a given graph such that subdivision paths satisfy a prescribed modular length constraint.  
\end{ques}\label{Q:minor} 
Note that, in a minor or a subdivision of a complete graph, the average degree may be only slightly above two, and the girth may be very large. Alon and Krivelevich \cite{AlonK} provided a positive answer to this question for a given subcubic graph by showing that.

\begin{thm}[\cite{AlonK}]\label{alonkre}
    For every subcubic graph $H$ and every integer $q \ge 2$, there exists
an  integer $f$  such that every $K_f$-minor contains a $q$-divisible
subdivision of $H$ as a subgraph. 
\end{thm}
Theorem \ref{alonkre} is qualitatively optimal. Indeed, a $K_f$-minor may have maximum degree 3, and the distance between every pair of vertices of degree 3 is always divisible by $q$. As a result, one cannot find a subdivision of any graph with maximum degree at least 4, nor a subdivision of a 3-regular graph that contains a subdivision path whose length is not divisible by $q$. 
For a subcubic graph $H$ and an integer $q \ge 2$, define $f(H,q)$ to be the smallest integer such that every graph containing a $K_f$-minor also contains a $q$-divisible subdivision of $H$.
Das, Draganić, and Steiner \cite{SNR} established that $f(H,q) \le 7mq + 8n + 14q$ for any subcubic graph $H$ with $n$ vertices and $m$ edges. This is a linear upper bound. Indeed, since every subdivision path must have length at least $q$ in a $q$-divisible subdivision of $H$, such a subdivision must contain at least $m(q - 1) + n$ vertices. Therefore, for $f = m(q - 1) + n - 1$, the complete graph $K_f$ is itself a $K_f$-minor that does not contain a $q$-divisible subdivision of $H$. In the absence of constructions yielding a better lower bound, it is natural to conjecture that this straightforward bound may be optimal. However, due to the considerable structural complexity of $K_f$-minors, it is often more tractable to restrict the class of host graphs to $K_f$-subdivisions. Motivated by this reason, Das, Draganić, and Steiner proposed the following problem in \cite{SNR}.

\begin{prob}[\cite{SNR}]\label{Prob:DDS}
    Given $q\in\mathbb{N}$ and a subcubic graph $H$ with $n$ vertices and $m$ edges, is it true for $f = m(q -  1) + n$ that every subdivision of $K_f$ contains a $q$-divisible subdivision of $H$? 
\end{prob}

Let $\mathbb{Z}_q$ denote the cyclic group of order $q$. A subdivision of $K_f$ induces a weight function $\omega : E(K_f) \to \mathbb{Z}_q$, defined by setting $\omega(e)$ to be the length modulo $q$ of the subdivision path corresponding to the edge $e \in E(K_f)$. In this setting, the graph $K_f$ equipped with such a weighting is referred to as a $\mathbb{Z}_q$-edge-weighted $K_f$. For a path $P$ in an edge-weighted graph, its weight, denoted by $\omega(P)$, is defined as the sum of the weights of all edges in $P$. In this paper, a subdivision $\sub(H)$ of a graph $H$ in a $\mathbb{Z}_q$-edge-weighted $K_f$ is called a {\em $q$-divisible subdivision} of $H$ if every subdivision path in $\sub(H)$ has weight congruent to zero modulo $q$.

For an integer $q \ge 2$ and a graph $H$, the {\em $q$-divisible subdivision number} of $H$ is defined as $$s_q(H)=\min\{f : \text{Every $\mathbb{Z}_q$-edge-weighted $K_f$ contains a $q$-divisible subdivision of $H$}\}.$$ 
When all edges of $K_f$ have weight 1, every subdivision path in a $q$-divisible subdivision $\sub(H)$ of $H$ has length at least $q$. This implies $|V(\sub(H))|\ge m(q - 1) + n$. Hence $s_q(H)\ge m(q-1)+n$. 
In this terminology, the assertion that $s_q(H) = m(q-1) + n$ for a subcubic graph $H$ with $n$ vertices and $m$ edges implies Problem~\ref{Prob:DDS} (Note that in the original setting, a branch vertex of degree at most two in the subdivision of $H$ may lie on a subdivision path of a $K_f$-subdivision.).

In this paper, we resolve Problem~\ref{Prob:DDS} for $q=2$. In fact, we prove a more general result: we determine the exact value of $s_2(H)$ when $H$ is any 5-degenerate graph.
\begin{restatable}{thm}{Degenerate}\label{qis2}
 Let $H$ be a 5-degenerate graph with $n$ vertices and $m$ edges. Then $s_2(H)=n+m$.
\end{restatable}

For general $q\ge 3  $, we address Problem~\ref{Prob:DDS} for trees. Specifically, we prove the following theorem.

\begin{restatable}{thm}{Tree}\label{THM:Tree}
    Let $T$ be a tree with $n$ vertices, and $q\ge 2$ be an integer. Then $s_q(T)=nq-q+1$. 
\end{restatable}

For general $q$ and $H$, we have the following result. 

\begin{restatable}{thm}{Generalq}\label{generalq}
For any graph $H$ with $n$ vertices and $m$ edges, and for any integer $q\ge 2$, we have
    $s_q(H)\le (2q - 1)m + 2n - 1 + 4q$. 
\end{restatable}

By induction on the edges of $H$ based on  Theorem \ref{THM:Tree}, we can obtain an improved upper bound when $q$ is a prime.

\begin{restatable}{thm}{Prime}\label{primeorder}
    Let H be a connected graph with $n$ vertices and $m$ edges. Let $p \ge 3$ be a prime. Then $s_p(H) \le \frac{3p - 1}{2}m - \frac{p - 1}{2}n + \frac{p + 1}{2}$. 
\end{restatable}

A {\em $t$-subdivision} of a graph $H$, denoted by $\sub_t(H)$, is the graph obtained by replacing each edge of $H$ with a path of length exactly $t+1$.
A natural question is what is the minimum number $f$ such that every $Z_q$-edge-weighted $K_f$ contains a $q$-divisible $\sub_t(H)$? 
Formally, for positive integers $q\ge 2$, $t$ with $q|(t+1)$, and a graph $H$, the {\em $q$-divisible $t$-subdivision number} of $H$ is defined  as 
$$s_q(H,t)=\min\{f : \text{Every $\mathbb{Z}_q$-edge-weighted $K_f$ contains a $q$-divisible $\sub_t(H)$}\}.$$
\noindent{\bf Remark:} (1) The condition $q|(t+1)$ is necessary because no $q$-divisible $\sub_t(H)$ exists in the $\mathbb{Z}_q$-edge-weighted $K_{f}$ in which all edges are assigned weight 1.

(2) The existence of $s_q(H,t)$ can be guaranteed by the Ramsey Theorem.

For a given graph $H$ with $n$ vertices and $m$ edges, $|V(\sub_t(H))|=n+tm$. Hence $s_q(H,t)\ge n+tm$. Unlike the case for $s_q(H)$, we have reason to believe $s_q(H)=n+(q-1)m$, while for $q\ge 3$, we have $s_q(H,t)> n+tm$. 
Indeed, if we fix a vertex $x$ in $K_{n+tm}$, assign 1 to all edges adjacent to $x$, and assign 0 to all the remaining edges, then we cannot find a $q$-divisible $\sub_t(H)$ in this $\mathbb{Z}_q$-edge-weighted $K_{n+tm}$.
In this paper, we further determine $s_2(H, 1)$ for trees and cycles.

\begin{restatable}{thm}{onesubTree}\label{ts}
 (1)   Let $T$ be a tree with $n$ vertices. Then $s_2(T,1)=2n-1$. 

 (2) Let $C$ be a cycle of order $n$. Then $s_2(C,1)=2n$.

\end{restatable}

In Section 2, we present preliminary notions and results necessary for the proofs. Section 3 is devoted to the proofs of Theorems \ref{generalq}, \ref{THM:Tree}, and \ref{primeorder}. In Section 4, we establish Theorem \ref{ts}. Finally, we provide some concluding remarks in the last section.

\section{Preliminaries}

We introduce more definitions and notation used in this article.  Let $V(G)$ and $E(G)$ denote the vertex set and edge set of a graph $G$, respectively.   For any two disjoint subsets $X,Y\subseteq V(G)$, let $G[X, Y]$ denote the induced bipartite subgraph of $G$ with vertex classes $X$ and $Y$. Define a clique of order $n$ as an $n$-clique.  
Let $(A, + )$ be an additive group. For subsets $S_1,S_2\subseteq A$, define their sum set as $S_1 + S_2 = \{s_1 + s_2| s_1\in S_1,s_2\in S_2\}$. More generally, define $\sum_{i = 1}^kS_i = \{s_1 + s_2 + \dotsb + s_k |s_i\in S_i,1\le i \le k\}$ for $S_i\subseteq A$, $1\le i\le k$. We write 
$B\le A$ to indicate that $B$ is a subgroup   of $A$, and $B<A$ to denote that $B$ is a proper subgroup of $A$. For $c_1,c_2,\dotsc,c_s\in A$, the subgroup of $A$ they generate is denoted by $\langle c_1,c_2,\dotsc,c_s\rangle$. With a slight ambiguity in the notation used, we will use the original graph's vertex notation to stand in for its corresponding branch vertex within the subdivision. Two edge-weighted graphs $G$ and $H$ are isomorphic if there exists an isomorphism between them that preserves the edge weights.

Let $(A, +)$ be a finite additive group. In the rest of this section, we write $K_f$ for an edge-weighted complete graph $K_f$ with weight function $\omega: E(K_f) \to A$ for convenience.




\begin{defi}[Efficient $t$-clique]
Let $K_t$ be a $t$-clique in $K_f$ with two designated vertices $x$ and $x'$.  The clique $K_t$ is said to be efficient if there exist  $t-1$ paths between $x$ and $x'$ in  $C$ which have distinct weights. We denote this structure by $xK_tx'$.
\end{defi}

\begin{defi}[$t$-connector]\label{DEF:Hconnector} 
 A  $t$-connector is a sequence formed by concatenating  efficient $t$-cliques, denoted as $x_1K_tx_{2}K_tx_3\dots x_{s}K_tx_{s+1}$, where each $x_iK_tx_{i+1}$ is an efficient $t$-clique. In this sequence, any two consecutive cliques $x_{i-1}K_tx_{i}$ and $x_iK_tx_{i+1}$ intersect at exactly one vertex  $x_i$, while any two cliques $x_{i}K_tx_{i + 1}$ and $x_jK_tx_{j+1}$ are disjoint if $|i-j|\ge 2$.
The vertices $x_1$ and $x_{s+1}$ are called the endpoints of the connector. 
The path $P=x_1x_2\dotsm x_sx_{s + 1}$ is called the base path. 
\end{defi}

\begin{defi}[$S$-connector]\label{DEF:connector} An efficient $xK_3x'$ with vertex set $\{x, y, x'\}$ is referred to $c$-efficient if $\omega(xy)+\omega(yx')-\omega(xx')=c$ (note that $c\not=0$ by the definition of the efficient $3$-clique). For convenience, we also write this as a $c$-efficient triple $(x,y,x')$.
A  $3$-connector $x_1K_3x_2K_3x_3\dots x_sK_3x_{s+1}$ with each $x_iK_3x_{i+1}$ being $c_i$-efficient
is also called an $S$-connector for $S \subseteq \sum_{i = 1}^s\{0,c_i\}$.

\end{defi}

 \begin{obs}\label{OBS:switching}
 If we have an $S$-connector with base path $P=x_1x_2\dotsm x_sx_{s + 1}$ as described in Definition~\ref{DEF:connector}, then for every $c_i\in S$, by switching  the two  paths of distinct weights between  $x_i$ and $x_{i + 1}$, we can obtain a path $Q$ of weight $\omega(Q) = \omega(P) + c_i$. Consequently, for every $c \in S$, one can obtain a path $Q$ with weight $\omega(Q) = \omega(P) + c$ through a sequence of such suitable  switches.

 \end{obs}

\begin{thm}[The Cauchy-Davenport Theorem, see \cite{Daven}]\label{THM: CDT}
    Let $p$ be a prime and let $A,B\subseteq \mathbb{Z}_p$, then  $|A + B| \ge \min\{p,|A| + |B| -1\}$.
\end{thm}

By the Cauchy-Davenport Theorem, we have the following observation.
\begin{obs}\label{Obs: K4-connector}
If $A = \mathbb{Z}_p$,  a  $4$-connector $X$ with $3s + 1$ vertices has at least $\min\{2s + 1,p\}$ paths between the endpoints $x_1$ and $x_{s + 1}$ with distinct weights. Let $R(X)$ be the set of the weights of the paths between $x_1$ and $x_{s + 1}$ in $X$.     
\end{obs}

The following definition comes from \cite{SNR}. 
\begin{defi}[$B$-restricted $K_f$]
Let $B \le A$. The edge-weighted complete graph $K_f$ is said to be $B$-restricted if it satisfies the following two conditions: 

    (1)  $2\omega(e) \in B$ for every edge $e \in E(K_f)$;
    
    (2) Whenever $xK_3y$ is a $c$-efficient  triangle, $c \in B$.  
\end{defi}

Note that if $G$ is $B$-restricted, then any subgraph of $G$ is $B$-restricted too.

\begin{lem}\label{res:cycle}
   If $K_f$ is $B$-restricted, then the weight of every cycle in $K_f$ is in $B$.
\end{lem}

\begin{proof}
    We proceed by induction. Note that $0\in B$ as well. First consider a triangle $xyz$; we have 
    \begin{align*}
        \omega(xy) + \omega(yz) - \omega(xz)\in B,\\
        \omega(xz) + \omega(yz) - \omega(xy)\in B,\\
        \omega(xy) + \omega(xz) - \omega(yz)\in B.
    \end{align*}
   Therefore, the sum of the left-hand sides, $\omega(xy) + \omega(yz) + \omega(xz)$, also  in $B$. Suppose all cycles of length  at most $k\ge 3$ have weights in $B$. Now consider a cycle of length $k + 1$ with total weight $a$.  Take any chord of weight $c$ in this cycle; it decomposes the cycle into two smaller cycles with weights $c_1\in B$ and $c_2\in B$, respectively. Since $2c\in B$, we have $a = c_1 + c_2 - 2c\in B$. Hence, every cycle in $K_f$ has a weight in $B$. 
\end{proof}

Let $P$ be a path with ends $x$ and $y$.  For two additional vertices $u$ and $v$, we define $uPv$ as the path formed by joining $u$ to $x$ and $v$ to $y$, i.e., $V(uPv) = V(P) \cup \{u, v\}$ and $E(uPv) = E(P) \cup \{xu, yv\}$.

\begin{lem}\label{zeropath}
Suppose $B\le A$.  If $K_f$ is $B$-restricted and $F$ is a $B$-connector, then for every edge $uv$ in $K_f - F$ with $\omega(uv) \in B$, there exists a zero-weight path between $u$ and $v$ whose internal vertices all lie in $F$. 
\end{lem}

\begin{proof}
   Let $x,y$ be the endpoints of $F$.  For the base path $P$ between $x$ and $y$ in $F$,  the cycle $uPvu$ has weight in $B$ since $K_f$ is $B$-restricted by Lemma \ref{res:cycle}. Therefore,  $uPv$ has  weight $b\in B$, given that $\omega(uv)\in B$. By Observation~\ref{OBS:switching},  we can obtain a new path $Q$ between $x$ and $y$ within $F$ with $\omega(Q) = \omega(P) - b$ as $-b\in B$. Therefore, the path $uQv$ has weight $\omega(ux)+\omega(Q)+\omega(yv)=\omega(uPv)-b=0$, which is a desired zero-weight path. 
\end{proof}
Let $uFv$ denote a zero-weight path stated in the above lemma. 

\begin{lem}\label{local}
 Fix a vertex $v$ in $K_f$. If there exists no $c$-efficient  $vK_3x$ in $K_f$, then 
define $B = \{0\}$. Otherwise, define
 \[
 B = \langle\{c | \text{ $(v,x,y)$ is a $c$-efficient triple, $x,y\in V(K_f)\setminus\{v\}$}\}\rangle.
 \]
  Then $K_f - v$ is $B$-restricted. 
\end{lem}

\begin{proof}
    For every edge $xy$ in $K_f - v$, if $(v,x, y)$ forms a $c$-efficient triple, then $\omega(vx) + \omega(xy) - \omega(vy)=c\in B$; otherwise, $\omega(vx) + \omega(xy) - \omega(vy)=0\in B$. 
 
    (1) For every edge $e=xy$ in $K_f - v$, by the symmetry of $x$ and $y$, we have 
    \[
    \omega(vx) + \omega(xy) - \omega(vy)\in B,  \text{ and }    \omega(vy) + \omega(xy) - \omega(vx)\in B. \]
    Therefore, the sum $2\omega(xy)\in B$. 

    (2) For every triple $(x,y,z)$ in $K_f - v$, we also have 
    \[
    \omega(vx) + \omega(xy) - \omega(vy)\in B,  \,    \omega(vy) + \omega(yz) - \omega(vz)\in B, \text{ and }\omega(vz) + \omega(xz) - \omega(vx)\in B.
    \]
    Therefore, the sum $\omega(xy) + \omega(yz) + \omega(xz)\in B$. By (1), we have $2\omega(xz)\in B$. Thus, $\omega(xy) + \omega(yz) - \omega(xz)\in B$.

    By (1) and (2), we have $K_f - v$ is $B$-restricted. 
\end{proof}

There exist four possible types $T_1$, $T_2$, $T_3$, $T_4$ of triangles within a $\mathbb{Z}_2$-edge-weighted graph $K_f$, as shown in Figure~\ref{Fig:triangletypes}. Let $T_i(u,v,w)$ $(1\le i\le 4)$ denote the corresponding triangle induced by vertices $u,v,w$ in a $\mathbb{Z}_2$-edge-weighted graph.  

\begin{figure}[ht]
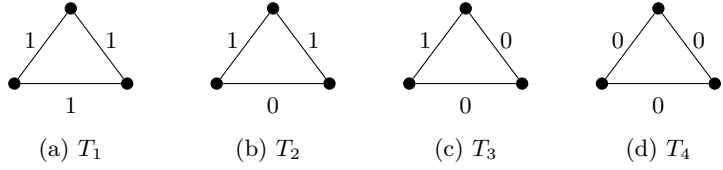

    \centering
    \subfloat[$T_1$]{\trif} \qquad \subfloat[$T_2$]{\tris} \qquad \subfloat[$T_3$]{\trit} \qquad \subfloat[$T_4$]{\trifo} 
    \caption{Four possible types of triangles within a $\mathbb{Z}_2$-edge-weighted graph.}\label{Fig:triangletypes}
\end{figure}

It should be noted that the vertices from $T_1$ or $T_3$ with any order form a $1$-efficient triple. When a $\mathbb{Z}_2$-edge-weighted $K_f$ contains only $T_2$ and $T_4$, then $K_f$ is $\{0\}$-restricted. The following observation can be directly verified from Observation~\ref{OBS:switching}.
\begin{obs}\label{spetri}
    For any $Z_2$-edge-weighted  $K_f$, if it contains an induced subgraph $T$ which is either $T_1$ or $T_3$, then for every pair of vertices $u,v\in V(K_f)$, there exists a zero-weight $uv$-path whose internal vertices lie entirely within $T$.
\end{obs}
For every pair of vertices $(u, v)$ and each triangle $T \in \{T_1, T_3\}$ in $K_f$, let $uTv$ denote the zero-weight path between $u$ and $v$ given by Observation \ref{spetri}. 
\begin{lem}\label{bipar}
    Let $f\ge 3$ be a positive integer. For every $\mathbb{Z}_2$-edge-weighted complete graph $K_f$ that does not contain a triangle isomorphic to either $T_1$ or $T_3$, there exists a partition $V_1,V_2$ of $V(K_f)$ such that:
\begin{itemize}[leftmargin = \parindent,nosep]
    \item Every edge between $V_1$ and $V_2$ has weight 1.
    \item Every edge within $V_1$ or within $V_2$ has weight 0.
\end{itemize}   
\end{lem}
\begin{proof}
    The absence of $T_3$-type triangle implies the following transitive property: Given any $u,v,w\in V(K_f)$,  $\omega(uv)=\omega(vw) = 0$ implies $\omega(uw)=0$. Consequently, the zero-weight edges induce a union of cliques. Moreover, the exclusion of $T_1$-type triangles limits the number of such maximal cliques to at most two, which immediately implies the result.
\end{proof}
Recall that a $\mathbb{Z}_2$-edge-weighted complete graph $K_f$ is $\{0\}$-restricted if and only if it contains neither $T_1$ nor $T_3$.

\begin{lem}\label{cliquezero}
   Let $A=\mathbb{Z}_q$, $B \le \mathbb{Z}_q$. Any $B$-restricted complete graph $K_{2n-1}$ must contain an $n$-clique in which every edge has a weight belonging to $B$.  
\end{lem}

\begin{proof}
    Let $B' := \{a\in A \,|\, 2a\in B\}$. Since $K_{2n - 1}$ is $B$-restricted, we have that every edge $e$ satisfies $\omega(e)\in B'$.  Note that $B$ is a subgroup of $B'$, we can define the quotient group $B^* = B'/B$ and, we have $|B^*| = \frac{|B'|}{|B|}\le 2$ since $A = \mathbb{Z}_q$. If $|B^*| = 1$, then $B' = B$, we are done! Otherwise, let $B^* = \{\bar{0},\bar{1}\}$. Reassign every edge of $K_f$ with weight the corresponding value in $B^*$. Then, since $K_{2n-1}$ is $B$-restricted, it is also $\{\bar{0}\}$-restricted with the corresponding weights in $B^*$. By Lemma \ref{bipar}, there exists an $n$-clique in which every edge has weight $\bar{0}$. Hence, in the original $B$-restricted graph $K_{2n-1}$, there also exists an $n$-clique with all edge weights in $B$.
\end{proof}

\section{For general $q$}

First, we give a Lemma about how to construct a (3-)connector. 

\begin{lem}[Connector construction]
    Let   $B\le A$, and let $G$ be a $B$-restricted $A$-edge-weighted $K_f$. Then at least one of the following two claims holds:

    (1) $G$ contains a $B$-connector $F$ with at most $2|B| - 1$ vertices.

    (2) There is a proper subgroup $B' < B$ and $G' = K_{f'} \subseteq K_f$ such that $G'$ is $B'$-restricted, where $f'\ge f - (2|B| - 3)$.  
\end{lem}
\begin{proof}
    We may assume that $f > 2|B| - 1$, as the second claim is trivially satisfied if we otherwise let $G'$ be a trivial subgraph of $G$. We will endeavor to create the connector by identifying a succession of subsets \( S_0 \subsetneq S_1 \dotsb \subsetneq S_t = B \) for some \( t \le |B| - 1 \), and for each \( i \in [t] \), an \( S_i \)-connector including  \( 2i + 1 \) vertices. For each $0\le i \le t - 1$, the $S_{i + 1}$-connector will augment the previously established $S_i$-connector by incorporating  2 new vertices. 

    Define $S_0 = \{0\}$, starting with any vertex. Assume that for some $0 \leq i \leq t - 1$, we have established an $S_i$-connector with $S_i \subsetneq B$, and we will now attempt to construct an $S_{i + 1}$-connector. Examine the graph derived from $G$ by eliminating all vertices in the $S_{i}$-connector; denote this graph as $G_i$. 

    Let $x,y\in G_i$ such that $(v,x,y)$ constitutes a $c$-efficient triple, and $S_i + \{0,c\} \neq S_i$. We finalize the design of the $S_{i + 1}$-connector by including the triangle $vxy$ into the $S_i$-connector, defining $S_{i + 1} = S_i + \{0,c\}$. By iterating this method, we may get the desired $B$-connector after a maximum of $|B| - 1$ steps, utilizing at most $2|B| - 1$ vertices. Conversely, if an iteration is unsuccessful, it must occur $S_{i} = S_i + \{0,c\}$ for some $0\le i\le |B| - 2$. Thus, $\langle c\rangle\subseteq S_i$ as $0\in S_i$. If no two vertices 
$x,y\in G_i$ form a $c$-efficient triple $(v,x,y)$, define $B' = \{0\}$. Otherwise, define
 \[
 B' = \langle\{c\, |\, \text{$(v,x,y)$ is a $c$-efficient triple, $x,y\in G_i$}\}\rangle.
 \]
     Then $G_i$ is $B'$-restricted for $B' \subseteq S_i$ and $B' < B$ by Lemma \ref{local}. Clearly, $V(G_i) \ge f - (2i + 1) \ge f - (2|B| - 3)$.   
\end{proof}

Before presenting the proof of Theorem~\ref{generalq}, we restate it. 
\Generalq*
\begin{proof}
Let $G$ be a complete graph $K_f$  with $\omega: E(K_f)\to \mathbb{Z}_q$, where $f = (2q - 1)m + 2n - 1 + 4q$. The idea of the proof is to construct $m$ disjoint connectors in $G$, each including no more than $2q - 1$ vertices. We will attempt the subdivision paths through these connectors, allowing us to implement switches within the connectors to guarantee that each path of weight $0\in \mathbb{Z}_q$.

We begin by creating $\mathbb{Z}_q$-connectors for as long as possible. If, at any time, no additional connectors can be constructed, this will provide some structural insights into the edge weights in the residual graph. This information will enable us to transition to a subgroup $A' < \mathbb{Z}_q$, and we shall thereafter construct $A'$-connectors instead. This method is reiterated until the requisite number of connectors is achieved. We will utilize these alongside the remaining vertices to construct a $q$-divisible subdivision of $H$. 

We initially set $f_1 = f$, $B_1 = A$, and $G_1 = G$. For each iteration of our process, suppose  we have $i\ge 1$ and $f_i,B_i$ and $G_i$ such that $G_i\subseteq G$ is a $B_i$-restricted $\mathbb{Z}_q$-weighted $K_{f_i}$; we shall then find $f_{i + 1},B_{i + 1}$ and $G_{i + 1}$ as follows. 

(1) If we find a $B_i$-connector $F_i$ using at most $2|B_i| - 1$ vertices, then let $G_{i + 1}$ be the graph by removing the vertices in $B_i$-connector from $G_i$, set $B_{i + 1} = B_i$ and we let $f_{i + 1}$ be the number of the vertices in $G_{i + 1}$. Note that $G_{i + 1}$ is $B_{i + 1}$-restricted.

(2) If we obtain a proper group $B' < B_i$, and $G' = K_{f'}\subseteq K_{f_i}$ is $B'$-restricted graph, where $f'\ge f_i - (2|B_i| - 3)$, then we remain in the current iteration, but update $G_i = G'$, $B_i = B'$, and set $f_i = f'$. We repeat the process until we encounter the first case. 

In this inductive process, the size of the group $B_i$ we are working with shrinks by a factor of at least two each time we approach scenario (2). Case (2) is thus encountered at most $\lceil\log_2 q\rceil$ times. Furthermore,  the group $B_i$ has a maximum size of $2^{-j}q$ following $j$ instances of scenario (2). We can bound the total number of vertices lost in case (2) during the procedure by $\sum_{j\ge 0}^{\lceil\log_2 q\rceil}(2\cdot 2^{-j}q - 3) \le 4q$, since we lose no more than $2|B_i| - 3$ vertices each time we fall into case (2).  

We will finally have the needed $m$ connectors after meeting the first case $m$ times, as the number of occurrences of the second case is finite. By the time we reach the final graph $G_{m + 1}$, we still have at least $ f-m (2q- 1) - 4q\ge 2n - 1$ vertices, and for each $1\le i \le  m$ we have a $B_i$-connector $F_i\subseteq G_i$. 
There is an $n$-clique with all edges weighted in $B_m$ in $G_{m + 1}$ according to Lemma \ref{cliquezero}; select these $n$ vertices to be the branch vertices of a $q$-divisible subdivision of $H$.  Now, we begin to construct a $q$-divisible subdivision of $H$. For each $e_k = uv$ in the $n$-clique corresponding to the edge in $H$, we will need a zero-weight path that connects the two vertices and is distinct from one another. Assume that we have already built subdivision paths for $e_1,\dotsc,e_{k - 1}$ for some $k\in [m]$. For $e_k = uv$,  we have $\omega_{G_k}(e_k)\in B_m\le B_k$. By Lemma \ref{zeropath}, we establish a zero-weight path $uF_kv$, yielding the desired result.
\end{proof}

Next, we restate Theorem~\ref{THM:Tree} and present its proof. 
\Tree*

\begin{proof}
It suffices to prove that every $\mathbb{Z}_q$-edge-weighted $K_{nq-q+1}$ contains a $q$-divisible subdivision of $T$. We may assume $n\ge 2$, as the case $n = 1$ is trivial. Suppose to the contrary that there exists a $\mathbb{Z}_q$-edge-weighted $K_f$ containing no $q$-divisible subdivisions of $T$ where $f = nq - q + 1$. Let $T$ be a minimal counterexample, and suppose there exists a weight function $\omega: E(K_f) \to \mathbb{Z}_q$ such that $K_f$ contains no $q$-divisible subdivision of $T$.

For a path $P = x_1x_2\dotsm x_{k}$ with $k\ge 2$, define 
\[
R(P) = \{\omega(x_ix_{i + 1}\dotsm x_k)| 1\le i < k\}.
 \]
If $|R(P)| = |P|- 1$, and $0\notin R(P)$, we call $P$ a {\em nice path}. Note that a non-zero-weight edge is a nice path. Take $P = x_1x_2\dotsm x_k\,(k\ge 2)$ to be a longest nice path. Since $R(P) \subseteq \mathbb{Z}_q\setminus\{0\}$, we have $k\le q$. By the minimality of $T$, $K_f -V(P)$ contains a $q$-divisible subdivision $\sub(T-v)$ for a leaf $v$ in $T$. Let $u$ be the neighbor of $v$ in $T$. If $\omega(x_iu) = 0$ for some $x_i\in V(P)$ in $K_f$, then by  adding the zero-weight edge $x_iu$ to $\sub(T - v)$, we obtain a $q$-divisible subdivision of $T$, a contradiction.

Now assume $\omega(x_iu)\not= 0$ for all $x_i\in V(P)$. If $-\omega(x_ku)\in R(P)$, then there must exist some $i$ such that $\omega(x_ix_{i + 1}\dotsm x_k)=-\omega(x_ku)$. Thus by adding the zero-weight  path $x_ix_{i + 1}\dotsm x_ku$ to $\sub(T - v)$, we  obtain a $q$-divisible subdivision of $T$, a contradiction. Therefore, assume $-\omega(x_ku)\notin R(P)$. 
Then we must have $k< q$; otherwise, since $R(P) = k-1\ge q - 1$, it would follow that $R(P) = \mathbb{Z}_q \setminus \{0\}$, implying  that $-\omega(x_ku)\in R(P)$, a contradiction. 
Set path
\[
P' = x_1x_2\dotsm x_{k}u.
\]
Then $R(P') = (R(P) + \omega(x_ku))\cup \{\omega(x_ku)\}$. Thus $|R(P')| = |R(P)| + 1 = |P'| - 1$  and $0\notin R(P')$, resulting in a longer nice path, a contradiction to the maximality of $P$.

\end{proof}


When $p$ is a prime and $H$ is connected, we have an improved upper bound for $s_p(H)$.

\Prime*


\begin{proof}
By induction on $m$. The base case $m = n - 1$ comes from Theorem \ref{THM:Tree}. Assume $m\ge n$ and the result holds for all connected graphs with $m - 1$ edges.

We start with a single vertex to create a $4$-connector   $X$ by adding three vertices sequentially. 
Once two endpoints of $X$ have $p$ paths between them with distinct weights, we terminate this procedure.
Suppose $|X| = 3s + 1$ when the procedure halts. Then by Observation~\ref{Obs: K4-connector}, $2s + 1\le p$. Otherwise, the procedure terminates when $|X| \le 3(s - 1) + 1$. 

    Let us assume we have successfully constructed a 4-connector 
    \[
    X = 
    \{x_1,x_2,\dotsc,x_{t + 1}\}
    \cup 
    \{y_1,y_2,\dotsc,y_t\}\cup \{z_1,z_2,\dotsc,z_t\}
    \]
    in a $\mathbb{Z}_q$-edge-weighted $K_f$, where $x_i, y_i, z_i, x_{i+1}$ form an efficient  4-clique $x_iK_4x_{i+1}$ for $i\in [t]$ and $f = \frac{3p - 1}{2}m - \frac{p - 1}{2}n + \frac{p + 1}{2}$. 
     Let $U=V(K_f)\setminus X$. Then
    \begin{equation}\label{EQ:lowerU}
    |U| = f - 3t - 1 \ge \frac{3p - 1}{2}(m - 1) - \frac{p - 1}{2}n + \frac{p + 1}{2}
    \end{equation}
    since $2t + 1\le p$. By the induction hypothesis, $K_f[U]$ contains a $q$-divisible subdivision $\sub(H - e)$ of $H-e$ for some edge $e\in H$ (Ensure that the graph $H - e$ remains connected). If  $R(X) = \mathbb{Z}_p$, then for edge $e = uv\in E(H)$, there exists a path $P$ between $x_1$ and $x_{t + 1}$ in $K_f[X]$ of  weight $-\omega(x_1u) - \omega(x_{t + 1}v)$. Together with the two edges $x_1u$ and $x_{t + 1}v$, we obtain a zero-weight path $uPv$ connecting $u$ and $v$ through  $X$.   
    By adding the zero-weight path $uPv$ to $\sub(H-e)$, we obtained a $q$-divisible subdivision of $H$.  
    
    Now assume $|R(X)| < p$. Choose arbitrarily a triangle $xyz$ in $K_f[U]$. Let $D = K_f[\{x_{t + 1},x,y,z\}]$. Assume that no larger 4-connector can be obtained.

\begin{cla}\label{zerotwo}
    If among edges $xy,yz,xz$, exactly two edges have weight zero, then $\omega(x_{t + 1}x) = \omega(x_{t + 1}y) = \omega(x_{t + 1}z)$. 
\end{cla}   
\begin{proof}
 Assume $\omega(xz)\ne 0$.  
 Examine three $x_{t + 1}$-$x$ paths in $D$: $x_{t  + 1}x$, $x_{t + 1}yx$, and $x_{t + 1}yzx$, each possessing weights $\omega(x_{t + 1}x)$, $\omega(x_{t + 1}y)$, and $\omega(x_{t + 1}y) + \omega(xz)$, respectively. If the three values are not mutually distinct, then we get $\omega(x_{t + 1}x) = \omega(x_{t + 1}y)$ or $\omega(x_{t + 1}x) = \omega(x_{t + 1}y) + \omega(xz)$ since $\omega(xz)\ne 0$.  Likewise, analyze three $x_{t + 1}$-$y$ paths in $D$: $x_{t + 1}y$, $x_{t + 1}xy$, and $x_{t + 1}xzy$, each with weights $\omega(x_{t + 1}y)$, $\omega(x_{t + 1}x)$, and $\omega(x_{t + 1}x) + \omega(xz)$, respectively.  If the three values are not mutually distinct, then we get $\omega(x_{t + 1}x) = \omega(x_{t + 1}y)$ or $\omega(x_{t + 1}y) = \omega(x_{t + 1}x) + \omega(xz)$ since $\omega(xz)\ne 0$. $\omega(x_{t + 1}x) = \omega(x_{t + 1}y) + \omega(xz)$ and $\omega(x_{t + 1}y) = \omega(x_{t + 1}x) + \omega(xz)$ cannot simultaneously hold, as $2\omega(xz) \neq 0$ in $\mathbb{Z}_p$. Assume $\omega(x_{t + 1}x) \ne \omega(x_{t + 1}y) + \omega(xz)$. If $\omega(x_{t + 1}x)\ne \omega(x_{t + 1}y)$, then $D$ is an efficient clique $x_{t + 1}K_4x$. Thus, $X\cup \{x_{t + 1},x,y,z\}$ is a larger 4-connector with endpoints $x_1$ and $x$, a contradiction.  Thus $\omega(x_{t + 1}x) = \omega(x_{t + 1}y)$. 

Next, analyze another three $x_{t + 1}$-$x$ paths in $D$: $x_{t + 1}x$, $x_{t + 1}yzx$, and $x_{t + 1}zx$, each with weight $\omega(x_{t + 1}x)$, $\omega(x_{t + 1}x) + \omega(xz)$, and $\omega(x_{t + 1}z) + \omega(xz)$. If the three values are not mutually distinct, then  $\omega(x_{t + 1}x) = \omega(x_{t + 1}z)$ or  $\omega(x_{t + 1}x) = \omega(x_{t + 1}z) + \omega(xz)$. Next, analyze three $x_{t + 1}$-$z$ paths: $x_{t + 1}z$, $x_{t + 1}xz$, and $x_{t + 1}yz$, each with weight $\omega(x_{t + 1}z)$, $\omega(x_{t + 1}x)  + \omega(xz)$, and $\omega(x_{t + 1}x)$. If the three values are not mutually distinct, then  $\omega(x_{t + 1}x) = \omega(x_{t + 1}z)$ or  $\omega(x_{t + 1}z) = \omega(x_{t + 1}x) + \omega(xz)$ since $\omega(xz)\ne 0$. Since $\omega(x_{t + 1}x) = \omega(x_{t + 1}z) + \omega(xz)$ and $\omega(x_{t + 1}z) = \omega(x_{t + 1}x) + \omega(xz)$ cannot simultaneously hold, we can assume $\omega(x_{t + 1}x)\ne \omega(x_{t + 1}z) + \omega(xz)$. If $\omega(x_{t + 1}x)\ne \omega(x_{t + 1}z)$, then $D$ is an efficient clique $x_{t + 1}K_4x$. Thus,  $X\cup \{x_{t + 1},x,y,z\}$ would form a larger 4-connector with endpoints $x_1$ and $x$, a contradiction.  Thus $\omega(x_{t + 1}x) = \omega(x_{t + 1}z)$.

Therefore, $\omega(x_{t + 1}x) = \omega(x_{t + 1}y) = \omega(x_{t + 1}z)$. 

\end{proof}

\begin{cla}\label{zeroone}
   If among edges $xy,yz,xz$, exactly  one   has weight zero, assume $\omega(xy) =  0$, then $\omega(x_{t + 1}x) = \omega(x_{t + 1}y) \ne \omega(x_{t + 1}z)$.
\end{cla}  
\begin{proof}
 
If $\omega(xz) + \omega(yz) = 0$, there exist five $x_{t + 1}$-$z$ paths: $x_{t + 1}xz$, $x_{t + 1}xyz$, $x_{t + 1}yz$, $x_{t + 1}yxz$, and $x_{t + 1}z$, with corresponding weights $\omega(x_{t + 1}x) + \omega(xz)$, $\omega(x_{t + 1}x) - \omega(xz)$, $\omega(x_{t + 1}y) - \omega(xz)$, $\omega(x_{t + 1}y) + \omega(xz)$, and $\omega(x_{t + 1}z)$, respectively. The set can be expressed as
\[
\left(\{\omega(x_{t + 1}x),\omega(x_{t + 1}y)\} + \{\omega(xz),-\omega(xz)\}\right)\cup \{\omega(x_{t + 1}z)\}. 
\]
By the Cauchy-Davenport Theorem, if $\omega(x_{t + 1}x)\ne \omega(x_{t + 1}y)$, there are three $x_{t + 1}$-$x$ paths with mutually distinct weights in $D$. Hence $X\cup \{x_{t + 1},x,y,z\}$ forms a larger 4-connector, a contradiction. Therefore, $\omega(x_{t + 1}x) = \omega(x_{t + 1}y)$. Further,  to prevent $D$ contains three $x_{t + 1}$-$x$ paths with mutually distinct weights, we must have   $\omega(x_{t + 1}z) = \omega(x_{t + 1}x) + \omega(xz)$ or $\omega(x_{t + 1}z) = \omega(x_{t + 1}x) - \omega(xz)$. Thus, $\omega(x_{t + 1}x) = \omega(x_{t + 1}y) \ne \omega(x_{t + 1}z)$. 

Now, assume $\omega(xz) +\omega(yz)\ne 0$. Consider three $x_{t + 1}$-$x$ paths: $x_{t + 1}x$, $x_{t + 1}yx$, and $x_{t + 1}yzx$ with weights $\omega(x_{t + 1}x)$, $\omega(x_{t + 1}y)$, and $\omega(x_{t + 1}y) + \omega(yz) + \omega(xz)$. If the three values are not mutually distinct, then $\omega(x_{t + 1}x) = \omega(x_{t + 1}y)$ or $\omega(x_{t + 1}x) = \omega(x_{t + 1}y) + \omega(yz) + \omega(xz)$. Similarly, if no three $x_{t+1}$-$y$ paths have pairwise distinct weights, then $\omega(x_{t + 1}x) = \omega(x_{t + 1}y)$ or $\omega(x_{t + 1}y) = \omega(x_{t + 1}x) + \omega(xz) + \omega(yz)$. 
Note that $\omega(x_{t + 1}x) = \omega(x_{t + 1}y) + \omega(yz) + \omega(xz)$ and $\omega(x_{t + 1}y) = \omega(x_{t + 1}x) + \omega(xz) + \omega(yz)$ cannot simultaneously hold. Assume $\omega(x_{t + 1}x) \ne \omega(x_{t + 1}y) + \omega(yz) + \omega(xz)$. Therefore, if $\omega(x_{t + 1}x) \ne \omega(x_{t + 1}y)$, there exist three $x_{t + 1}$-$x$ paths with mutually distinct weights, which enables us to create a larger 4-connector, resulting in a contradiction. Thus $\omega(x_{t + 1}x) = \omega(x_{t + 1}y)$. Next, consider four $x_{t + 1}$-$x$ paths: $x_{t + 1}x$, $x_{t + 1}yzx$, $x_{t + 1}zx$, and $x_{t + 1}zyx$ with weights $\omega(x_{t + 1}x)$, $\omega(x_{t + 1}x) +\omega(yz)+\omega(xz)$, $\omega(x_{t + 1}z) + \omega(xz)$, and $\omega(x_{t + 1}z) + \omega(yz)$. We already have $\omega(x_{t + 1}x) \ne \omega(x_{t + 1}x) + \omega(xz) + \omega(yz)$. To prevent there are three $x_{t + 1}$-$x$ paths with  mutually distinct weights, we have 
$$    \omega(x_{t + 1}x) = \omega(x_{t + 1}z) + \omega(xz) \text{~or~} \omega(x_{t + 1}x) + \omega(xz) + \omega(yz) = \omega(x_{t + 1}z) + \omega(xz),$$
and
$$    \omega(x_{t + 1}x) = \omega(x_{t + 1}z) + \omega(yz) \text{~or~} \omega(x_{t + 1}x) + \omega(xz) + \omega(yz) = \omega(x_{t + 1}z) + \omega(yz).$$
Since $\omega(xz),\omega(yz)\ne 0$, we conclude that $\omega(xz) = \omega(yz)$, and  $\omega(x_{t + 1}z) = \omega(x_{t + 1}x) - \omega(xz)$ or $\omega(x_{t + 1}x) + \omega(xz)$. Therefore, $\omega(x_{t + 1}x) = \omega(x_{t + 1}y) \ne \omega(x_{t + 1}z)$. 
 
\end{proof}

\begin{cla}\label{nonzero}
    $\omega(xy),\omega(yz),\omega(xz)$ can not be all non-zero. 
\end{cla}

\begin{proof}
Suppose not. 
{First, suppose $\omega(xy)\ne \omega(xz) + \omega(yz)$.} Consider paths between $x_{t + 1}$ and $x$: $x_{t + 1}x$, $x_{t + 1}yx$, and $x_{t + 1}yzx$ with weights $\omega(x_{t + 1}x)$, $\omega(x_{t + 1}y) + \omega(xy)$, and $\omega(x_{t + 1}y) + \omega(xz) + \omega(yz)$; and paths between $x_{t + 1}$ and $y$ in $D$: $x_{t + 1}y$, $x_{t + 1}xy$, and $x_{t + 1}xzy$ with weights $\omega(x_{t + 1}y)$, $\omega(x_{t + 1}x) + \omega(xy)$, and $\omega(x_{t + 1}x) + \omega(xz) + \omega(yz)$. If neither of the two sets of three values consists of all distinct values, we can conclude that  
$$    \omega(x_{t + 1}x)= \omega(x_{t + 1}y) + \omega(xy) \text{~or~} \omega(x_{t + 1}x) = \omega(x_{t + 1}y) + \omega(xz) + \omega(yz),$$
and 
$$\omega(x_{t + 1}y)= \omega(x_{t + 1}x) + \omega(xy) \text{~or~} \omega(x_{t + 1}y) = \omega(x_{t + 1}x) + \omega(xz) + \omega(yz)$$
by $\omega(xy) \ne \omega(xz) + \omega(yz)$. 
Since  $\omega(xy)\ne 0$, we have either
\[\omega(x_{t + 1}x) = \omega(x_{t + 1}y) \text{~and~} \omega(xz) + \omega(yz) = 0,\] 
or
 \[
 \omega(x_{t + 1}x) = \omega(x_{t + 1}y) + \omega(xy) \text{~and~} \omega(xz) + \omega(yz) + \omega(xy) = 0,
 \]
 or 
 \[
  \omega(x_{t + 1}y) = \omega(x_{t + 1}x) + \omega(xy) \text{~and~} \omega(xz) + \omega(yz) + \omega(xy) = 0,
 \]
If it is the first case, consider paths between $x_{t + 1}$ and $z$ in $D$: $x_{t + 1}xz$, $x_{t + 1}yz$, $x_{t + 1}xyz$, and $x_{t + 1}yxz$ with weights  $\{\omega(x_{t + 1}x)$, $\omega(x_{t + 1}x) + \omega(xy)\} + \{\omega(xz),-\omega(xz)\}$. We can conclude that $\omega(xy) = 0$ or $\omega(xz) = 0$  by the Cauchy-Davenport Theorem if there are no three distinct values, a contradiction. Thus, $\omega(xy) + \omega(yz) + \omega(xz) = 0$, and we can assume $\omega(x_{t + 1}y) = \omega(x_{t + 1}x) + \omega(xy)$. By the symmetry of $x, y$ and $z$, we have either $\omega(x_{t + 1}x) = \omega(x_{t + 1}z) + \omega(xz)$ or $\omega(x_{t + 1}x) = \omega(x_{t + 1}z) - \omega(xz)$. 
If $\omega(x_{t + 1}x) = \omega(x_{t + 1}z) - \omega(xz)$,  consider paths between $x_{t + 1}$ and $z$ in $D$: $x_{t + 1}z,x_{t + 1}yz$ and $x_{t + 1}yxz$ with weights $\omega(x_{t+1}z),\omega(x_{t + 1}y) + \omega(yz),\omega(x_{t + 1}y) + \omega(xy) + \omega(xz)$. Whenever two of them are equal, we always obtain one of $\omega(xy),\omega(xz)$ and $\omega(yz)$ is 0, a contradiction.      If $\omega(x_{t + 1}x) = \omega(x_{t + 1}z) + \omega(xz)$,  consider paths between $x_{t + 1}$ and $z$ in $D$: $x_{t + 1}z,x_{t + 1}xz$ and $x_{t + 1}yxz$ with weights $\omega(x_{t+1}z),\omega(x_{t + 1}x) + \omega(xz),\omega(x_{t + 1}y) + \omega(xy) + \omega(xz)$. Whenever two of them are equal, we always obtain one of $\omega(xy),\omega(xz)$ and $\omega(yz)$ is 0, a contradiction.

Therefore, by the symmetry of $x, y$ and $z$, we have
\[  
\left\{\begin{aligned}
    \omega(xy) = \omega(xz) + \omega(yz),\\
    \omega(yz) = \omega(xy) + \omega(xz),\\
    \omega(xz) = \omega(xy) + \omega(yz).
\end{aligned}\right. 
\]
By solving the system of equations, we obtain $\omega(xy) =  \omega(yz) = \omega(xz) = 0$, a contradiction. 
\end{proof}

Now we analyze the structure of $K_f[U]$. Choose a nonzero weight edge $xy$ in $K_f[U]$. If there exists a vertex $z$ with  $\omega(xz)\ne 0$ or $\omega(yz) \ne 0$, by Claim \ref{nonzero},  exactly one of $\omega(xz)$ and $\omega(yz)$ is 0. 
By Claim \ref{zeroone}, we have $\omega(x_{t + 1}x) \ne\omega(x_{t + 1}y)$. Thus,   by Claim \ref{zerotwo}, $\omega(xz')$ and $\omega(yz')$ can not be zero simultaneously for every $z'\in U\setminus\{x,y\}$. Therefore, $U\setminus\{x,y\}$ has a partition $ U_1\cup U_2$ such that each vertex in $U_1$ is incident to $x$ via an edge of nonzero weight, and to $y$ via an edge of weight 0; each vertex in  $U_2$ is incident to $y$ via an edge of nonzero weight, and to $x$ via an edge of weight 0. Moreover, by Claims \ref{nonzero},  all edges inside $U_1$ and $U_2$ have weight 0.
Consequently, the set $U$ has a partition $(U_1\cup \{y\})\cup (U_2\cup \{x\})$ such that all edges inside $U_1\cup \{y\}$ and $U_2\cup \{x\}$ have weight zero. 
By (\ref{EQ:lowerU}) and $m\ge n$, $|U| \ge np - p + 1\ge 2n - 1$. It follows that one of the sets $(U_1\cup \{y\})$ and $(U_2\cup \{x\})$ must have at least $n$ vertices, and hence we find a copy of $H$ contained in $K_f[U]$ with all edge weights equal to 0.  Now assume for every vertex $z\in U\setminus\{x,y\}$, $\omega(xz) = \omega(yz) = 0$. Then there is no nonzero edges between $\{x,y\}$ and $U\setminus \{x,y\}$. It follows that all nonzero weight edges form a matching in $K_f[U]$. Consequently, there exists a copy of $H$ in it with all edges of weight 0 as well. 
\end{proof}

\section{Proof of Theorem \ref{qis2}}
Given a $\mathbb{Z}_2$-edge-weighted $K_f$, for a vertex $v$, any neighbor incident to $v$  by an edge of weight $i\in\mathbb{Z}_2$ is referred to as an {\em $i$-neighbor} of $v$. 
Given a subset of vertices $U\subseteq V(K_f)$, we use $N_U^i(v)$ to denote the set of the $i$-neighbors of $v$ in $U$.
Given a vertex set $S$, let $N_U^i(S) = \bigcup_{v\in S} N_U^i(v)\setminus S$.     
Let $E^i(X,Y)$ be the set of  the edges between $X$ and $Y$ with weight $i\in\mathbb{Z}_2$. 
Now we present the proof of Theorem~\ref{qis2}

\Degenerate*
\begin{proof}
It suffices to prove that every $Z_2$-edge-weighted $K_{n+m}$ contains a 2-divisible subdivision of $H$.
We prove the theorem by induction on $n=|V(H)|$. The base case $n=1$ or 2 holds trivially. Now assume $n\ge 3$ and the statement is true for all $n'<n$. Let $H$ be a 5-degenerate graph with $n$ vertices and $m$ edges.
   We may assume  $H$ is connected.  Otherwise, by the induction hypothesis, we can disjointly embed a 2-divisible subdivision of each component of $H$ into $K_{n + m}$, yielding a  2-divisible subdivision of $H$.  Let $\omega: E(K_{n + m})\to\mathbb{Z}_2$ be an edge weight function on $K_{n + m}$. We denote the edge-weighted complete graph by $G$.

    If $G$ contains neither $T_1$ nor $T_3$, then by Lemma \ref{bipar},  there exists a partition $\{V_1,V_2\}$ of $V(G)$ with  
 all edges between $V_1$ and $V_2$ of weight 1 and all edges inside $V_1$ or $V_2$ of weight 0. Without loss of generality, assume $|V_1| \ge |V_2|$. Then $|V_1| \ge \lceil\frac{n + m}{2}\rceil \ge n=|V(H)|$ since $H$ is connected. Since every edge in $V_1$ has weight 0, $H$ can be embedded into $V_1$, which provides us with a 2-divisible subdivision of $H$. Therefore, we may always assume that $G$ contains at least one triangle isomorphic to either $T_1$ or $T_3$. 
Since $H$ is 5-degenerate, $\delta(H)\le 5$. We distinguish the following proof into five cases according to the minimum degree of $H$. 
    
\vspace{10pt}
   \noindent {\bf (1)} $\delta(H)=1$. 
   
   Let $u$ be a vertex of degree one in $H$ and $N_H(u)=\{v\}$. Let $H' = H - u$. Then $|V(H')|+|E(H')|= n + m - 2$. 
   If all edge weights of $G$ are 0, then the conclusion holds trivially. Otherwise, there exists an edge $u_1u_2$ of weight 1 in $G$.  By induction, $G - \{u_1,u_2\}$ contains  a 2-divisible subdivision $\sub(H')$ of $H'$. 
    For convenience, we will use the same notation for the vertices of the original graph and their corresponding branch vertices in the subdivision that follows. For the branch vertex $v$, if either $\omega(vu_1) = 0$ or $\omega(vu_2) = 0$, then by adding an edge of weight 0 to $\sub(H')$, we have a 2-divisible subdivision of $H$. 
   Now assume $\omega(vu_1) = \omega (vu_2) = 1$. Then by adding the zero-weight path $vu_1u_2$ or $vu_2u_1$ to $\sub(H')$, we obtain a 2-divisible subdivision of $H$. 

    \vspace{10pt}
    \noindent {\bf (2)} $\delta(H)=2$.
    
    Let $u$ be a vertex $u$ of degree 2 in $H$ and $N_H(u)=\{v_1,v_2\}$.  Let $H'=H - u$. Then $|V(H')|+|E(H')| = n + m  - 3$. 
   
   If $G$ contains a $T_1$-type triangle $T_1(u_1,u_2, u_3)$.  
        By the induction hypothesis, $G - \{u_1,u_2,u_3\}$ contains a 2-divisible subdivision $\sub(H')$ of $H'$. 
        If $\omega(v_iu_j)= 1$ for all $1\le i\le 2$ and $1\le j\le 3$,  by adding the zero-weight paths $v_1u_1u_3$ and $v_2u_2u_3$ to $\sub(H')$, we  obtain a 2-divisible subdivision of $H$. Otherwise, without loss of generality, assume $\omega(v_1u_1) = 0$. Then  by adding the edge $v_1u_1$ and a zero-weight path $v_2T_1(u_1,u_2,u_3)u_1$ guaranteed by Observation~\ref{spetri}, we obtain a 2-divisible subdivision of $H$.  
  
       Now assume $G$ contains no $T_1$-type triangles. Then $G$ must contain a $T_3$-type triangle. Let $T_3(u_1,u_2,u_3)$ be a $T_3$-type triangle with $\omega(u_1u_2) = 1$. 
        By the induction hypothesis, $G - \{u_1,u_2,u_3\}$ contains a 2-divisible subdivision $\sub(H')$ of $H-u$.  If $\omega(v_1u_1) = \omega(v_1u_2) = 1$, this results in a $T_1$-type triangle, a contradiction.  Thus, we may assume $\omega(v_1u_1) = 0$.
        By adding the edge $v_1u_1$ and a zero-weight path $v_2T_3(u_1,u_2,u_3)u_1$ guaranteed by Observation~\ref{spetri} to $\sub(H')$, we obtain a 2-divisible subdivision of $H$.

\vspace{10pt}

    \noindent {\bf (3)} $\delta(H)=3$.
    
 Let $u$ be a vertex of degree three in $H$ and $N_H(u)=\{v_1,v_2,v_3\}$. Let $H'=H - u$. Then $|V(H')|+ |E(H')| = n + m - 4$.

 Suppose $G$ contains a $T_1$-type triangle $T_1(u_1,u_2,w)$. If $G$ further contains an edge $wu_3$ with $\omega(wu_3) = 1$ for some $u_3\in V(G)\setminus \{u_1,u_2,w\}$, by the induction hypothesis, $G - \{u_1,u_2,u_3,w\}$ contains a 2-divisible subdivision $\sub(H')$ of $H'$.  
        If $|N^0_{\{v_1,v_2,v_3\}}(x)| = 3$, for some vertex $x\in \{u_1,u_2,u_3\}$, then  the addition of edges $\{xv_1,xv_2,xv_3\}$ to $\sub(H')$ results in a 2-divisible subdivision of $H$. 
       Thus $|N^0_{\{v_1,v_2,v_3\}}(x)|\le 2$ for any vertex $x\in \{u_1,u_2,u_3\}$.
        If there exists some vertex $x\in \{u_1,u_2,u_3\}$ with $|N^0_{\{v_1,v_2,v_3\}}(x)| = 2$, suppose $N^0_{\{v_1,v_2,v_3\}}(x)=\{v_1,v_2\}$, then by adding the edges $v_1x,v_2x $ and a zero-weight path $v_3T_1(u_1,u_2,w)x$ guaranteed by Observation~\ref{spetri} to $\sub(H')$, we thereby obtain a 2-divisible subdivision of $H$.  
        Therefore, we may assume  $|N^0_{\{v_1,v_2,v_3\}}(x)|\le 1$ (or equivalently, $|N^1_{\{v_1,v_2,v_3\}}(x)|\ge 2$) for any vertex $x\in \{u_1,u_2,u_3\}$. Let $B$ be the bipartite subgraph induced by all edges of weights 1 in the complete bipartite graph $G[\{v_1,v_2,v_3\},\{u_1,u_2,u_3\}]$. Then $d_B(x)\ge 2$ for any $x\in \{u_1,u_2,u_3\}$. If there exists a perfect matching in $B$,  then by adding this perfect matching and the edges $u_1w$, $u_2w$,  and $u_3w$ to $\sub(H')$, we obtain a 2-divisible subdivision of $H$.  By Hall's Theorem,  $B$ fails to have a perfect matching only if for each vertex $x\in \{u_1,u_2,u_3\}$, the sets $N^1_{\{v_1,v_2,v_3\}}(x)$ are identical and $|N^1_{\{v_1,v_2,v_3\}}(x)|=2$.
        Without loss of generality, assume $N^1_{\{v_1,v_2,v_3\}}(x)=\{v_2, v_3\}$ for  any vertex $x\in \{u_1,u_2,u_3\}$.
        Then $\omega(v_1u_1) = \omega(v_1u_2) = \omega(v_1u_3) = 0$.  If either $\omega(v_2w)$ or $\omega(v_3w)$ equals 1, say $\omega(v_3w)=1$, then by adding the edge $u_1v_1$, and the zero-weight paths $u_1u_2v_2$, $u_1wv_3$ to $\sub(H')$, we obtain a 2-divisible subdivision of $H$. Otherwise, $\omega(v_2w) = \omega(v_3w) = 0$. Consequently, by adding the edges $wv_2$, $wv_3$ and a zero-weight path $wT_1(u_1,u_2,w)v_1$ (guaranteed by Observation~\ref{spetri}) to $\sub(H')$, we  obtain a 2-divisible subdivision of $H$.
      Now assume $\omega(wv)=0$   for all $v\in V(K_f)\setminus\{u_1,u_2,w\}$. By the induction hypothesis, $K_f - \{u_1,u_2,w\}$ contains a 2-divisible subdivision $\sub(H')$ of $H'$. Note that $\omega(wv_1) = \omega(wv_2) = \omega(wv_3) = 0$. The addition of the edges $wv_1$, $wv_2$, and $wv_3$ to $\sub(H')$ yields a 2-divisible subdivision of $H$.

     Now, assume $G$ contains no $T_1$-type triangles and thus it must contain a $T_3$-type triangle. Suppose $T_3(u_1,u_2,u_3)$ is a $T_3$-type triangle with $\omega(u_1u_2) = 1$ in $G$.  Let $H'=H-u$. Then $|V(H')|+|E(H')|= n + m -4$. By the induction hypothesis, $G - \{u_1,u_2,u_3\}$ contains a 2-divisible subdivision $\sub(H')$  of $H'$. 
 If there exists a vertex $x\in \{u_1,u_2,u_3\}$ satisfying $|N^0_{\{v_1,v_2,v_3\}}(x)| = 3$, then  the addition of edges $\{xv_1,xv_2,xv_3\}$ to $\sub(H')$ results in a 2-divisible subdivision of $H$.
     Now assume $|N^0_{\{v_1,v_2,v_3\}}(x)|\le 2$ for any $x\in\{u_1,u_2,u_3\}$.
  If there exists a vertex, say $u_1$, with  $|N^0_{\{v_1,v_2,v_3\}}(u_1)|=2$, assume $N^0_{\{v_1,v_2,v_3\}}(u_1)=\{v_1, v_2\}$, then the addition of the edges $u_1v_1$, $u_1v_2$, and a zero-weight path $u_1T_3(u_1, u_2, u_3)v_3$ to $\sub(H')$ yields a 2-divisible subdivision of $H$. Therefore, we may assume $|N^0_{\{v_1,v_2,v_3\}}(x)|\le 1$   for any $x\in\{u_1,u_2,u_3\}$. Thus, there are at least four edges of weight 1 between $\{u_1, u_2\}$ and $\{v_1, v_2, v_3\}$. By the pigeonhole principle, there must be a vertex in $\{v_1,v_2,v_3\}$ incident to both $u_1$ and $u_2$ via edges of weight 1. This results in a $T_1$-type triangle, a contradiction.

    \begin{figure}[ht]
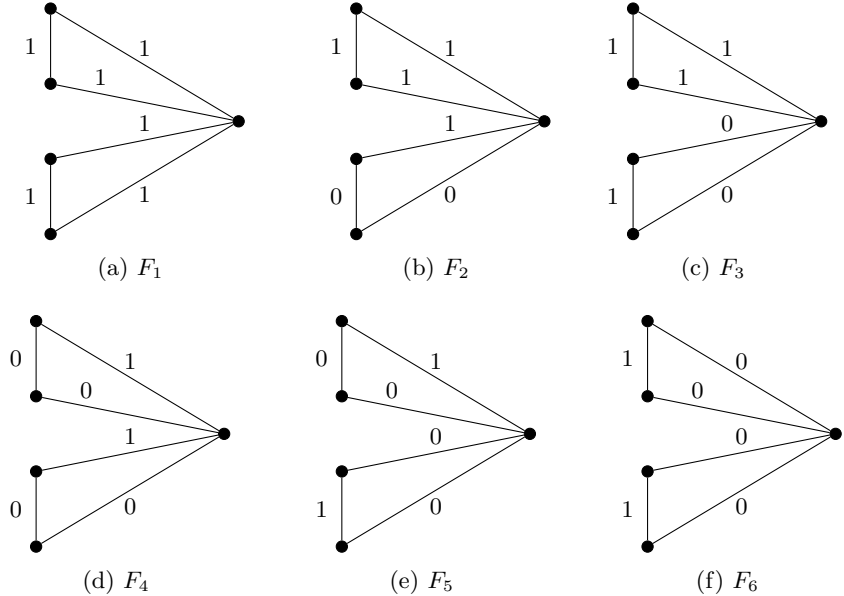

        \centering
        \subfloat[$F_1$]{\intertri{1}{1}{1}{1}{1}{1}} \qquad \subfloat[$F_2$]{\intertri{0}{1}{0}{1}{1}{1}}  \quad \subfloat[$F_3$]{\intertri{1}{1}{0}{0}{1}{1}} \\\subfloat[$F_4$]{\intertri{0}{0}{0}{1}{0}{1}} \qquad \subfloat[$F_5$]{\intertri{1}{0}{0}{0}{0}{1}} \qquad \subfloat[$F_6$]{\intertri{1}{1}{0}{0}{0}{0}} 
        \caption{Vertex-sharing triangles.}\label{Fig:TWOTRI}
    \end{figure}

\vspace{10pt}
    \noindent{\bf (4)} $\delta(H)=4$.
    
    Let $u$ be a vertex of degree four in $H$ and $N_H(u)=\{v_1,v_2,v_3,v_4\}$.  Let $H'=H - u$. Then $|V(H')|+|E(H')| = n + m - 5$.
        In this case, we consider structures composed of two triangles sharing a common vertex of either $T_1$-type or $T_3$-type, that yield six distinct types, denoted by $F_1, F_2, \dots, F_6$  (see Figure~\ref{Fig:TWOTRI}). Before this, we need to prove the existence of such a structure. 

\begin{cla}\label{CL:F}
     $G$ contains at least one of structures $F\in\{F_1, F_2, \dots, F_6\}$. 
\end{cla}
\begin{proof}
       Since $\delta(H) = 4$, we have $n \ge 5$, $m \ge 2n$, and $n + m \ge 2n + n = 3n$.  
Then, there must exist a triangle of type $T_1$ or $T_3$ in $K_f$. Remove one such triangle. This leaves at least $3n - 3$ vertices. If no further $T_1$ or $T_3$ exists in the remaining graph, then by Lemma \ref{bipar}, we can embed $H$ into $K_f$ using only edges of weight 0.
Therefore, we may assume there exist two disjoint triangles, each isomorphic to either $T_1$ or $T_3$. 
   
     If $G$ contains two disjoint $T_1$-type triangles, choose one vertex $u$ from one of these triangles, then among the three edges connecting $u$ to the vertices of the other triangle, by the pigeonhole principle, at least two of them have the same weight.  Selecting two of such edges,  we can then construct either an $F_1$ or an $F_3$.
     
   If $G$ contains two disjoint  triangles,  one being a $T_1$-type triangle $T_1(x_1,x_2,x_3)$ and the other  a $T_3$-type triangle $T_3(y_1,y_2,y_3)$ with $\omega(y_1y_2) = 1$.  If there exists a vertex  $x\in\{x_1,x_2,x_3\}$ such that the edges connecting $x$ to $\{y_1, y_2\}$ have the same weight, then we have either an $F_1$ or an $F_3$, we are done.
   Hence we may assume that for every vertex in $\{x_1,x_2,x_3\}$, the edges connecting it to $\{y_1, y_2\}$ have different weights. Without loss of generality, assume $\omega(x_1y_1) = 1$ and $\omega(x_1y_2) = 0$. Then whenever $\omega(x_1y_3) = 1$ or 0, we can construct an $F_2$ (for example, if $\omega(x_1y_3)=1$, then $T_1(x_1,x_2,x_3)$ and the triangle $T_3(x_1, y_2, y_3)$ construct an $F_5$).
  
   Finally, suppose that $G$ contains two disjoint $T_3$-type triangles,  one being $T_3(x_1,x_2,x_3)$ with $\omega(x_1x_2)=1$ and the other  $T_3(y_1,y_2,y_3)$ with $\omega(y_1y_2) = 1$. If the edges connecting $x_3$ to $\{y_1, y_2\}$ have the same weight, then we have either an $F_3$ or an $F_6$.
   Hence, we may assume that, without loss of generality, $\omega(x_3y_1) = 1$ and $\omega(x_3y_2) = 0$.  Then whenever $\omega(x_3y_3) = 1$ or 0, we can construct an $F_5$.
   
    \end{proof}

 According to Claim~\ref{CL:F}, $G$ contains an $F\in \{F_1, F_2, \dots, F_6\}$. Choose such an $F$, let $w$ be the shared vertex in $F$, and  $u_1u_2$ and $u_3u_4$ be the two edges not incident to $w$.  By the induction hypothesis, $G - F$ contains a 2-divisible subdivision  $\sub(H')$ of $H'$.
 Note that $F$ consists of two triangles belonging to  $\{T_1,T_3\}$. If       $|N_{\{v_1,v_2,v_3,v_4\}}^0(x)|\ge 3$ for some vertex $x\in\{u_1,u_2,u_3,u_4\}$, then, combining with Observation~\ref{spetri}, we always can construct four disjoint zero-weight paths connecting $x$ to $\{v_1,v_2,v_3,v_4\}$, which yields a 2-divisible subdivision of $H$.
Therefore, 
\begin{equation}\label{EQ:cl2}
  |N_{\{v_1,v_2,v_3,v_4\}}^0(x)|\le 2 \text{ (or equivalently  $|N_{\{v_1,v_2,v_3,v_4\}}^1(x)|\ge 2$) for any $x\in\{u_1,u_2,u_3,u_4\}$.}  
\end{equation}

If $\omega(u_1u_2) = 1$, we begin by considering the following \textbf{special case}.\\
\textbf{Special case}: If $|N_{\{v_1,v_2,v_3,v_4\}}^0(u_1)| =  2$ with $\omega(v_1u_1) = \omega(v_2u_1) = 0$,  then $\omega(u_2v_3) = \omega(u_2v_4) = 0$; otherwise, assume $\omega(u_2v_3)=1$, then by adding the edges $u_1v_1$, $u_1v_2$, the zero-weight path $u_1u_2v_3$, and the zero-weight path $u_1T(u_3,u_4,w)v_4$ to $\sub(H')$,  we obtain a 2-divisible subdivision of $H$. By (\ref{EQ:cl2}), we have $\omega(v_3u_1) = \omega(v_4u_1) = \omega(v_1u_2) = \omega(v_2u_2) = 1$. 

According to Claim~\ref{CL:F}, we distinguish this case into four subcases.

\vspace{5pt}
        \noindent\textbf{(4.1)} There exists an $F_1$ in $K_f$.

        If there exists a perfect matching in the bipartite graph $G[\{u_1,u_2,u_3,u_4\},\{v_1,v_2,v_3,v_4\}]$  with all edges of weight 1, denoted by $u_1v_{i_1}, u_2v_{i_2}, u_3v_{i_3}, u_4v_{i_4}$,  then by adding the zero-weight paths $wu_1v_{i_1}, wu_2v_{i_2}, wu_3v_{i_3}, wu_4v_{i_4}$ to $\sub(H')$, we  construct a 2-divisible subdivision of $H$.  
        Now, suppose that the bipartite graph $G[\{u_1,u_2,u_3,u_4\},\{v_1,v_2,v_3,v_4\}]$ contains no perfect matching with all edges of weight 1. By Hall's Theorem, there exists a subset  $S\subseteq \{v_1,v_2,v_3,v_4\}$ such that \[|N^1_{\{u_1,u_2,u_3,u_4\}}(S)| < |S|.\] 
     If $|S| = 4$, then there exists a vertex $x\in\{u_1,u_2,u_3,u_4\}$ satisfying $\omega(xv_i) = 0$ for $1\le i\le 4$, a  contradiction to (\ref{EQ:cl2}).
     If $|S| = 3$, then there are   two vertices in $\{u_1,u_2,u_3,u_4\}$ with the $1$-neighbors  $\{v_1,v_2,v_3,v_4\}$ at least $ 3$,  contradicting to (\ref{EQ:cl2}).
           If $|S| = 2$, then $|N^1_{\{u_1,u_2,u_3,u_4\}}(S)|\le 1$. Without loss of generality, assume $S = \{v_1,v_2\}$. Then $|N^0_{\{u_1,u_2,u_3,u_4\}}(v_1)\cap N^0_{\{u_1,u_2,u_3,u_4\}}(v_2)|\ge 3$. Without loss of generality,  assume 
            \[N^0_{\{u_1,u_2,u_3,u_4\}}(v_i) \supseteq \{u_1,u_2,u_3\}\]
            for $i = 1,2$. This implies that $|N^0_{\{v_1,v_2,v_3,v_4\}}(u_j)|\ge 2$ for $j = 1,2,3$.
            By (\ref{EQ:cl2}), we have $\omega(v_iu_j) = 1$ for all $3\le i \le 4$ and $1\le j\le 3$. This contradicts the structure of the \textbf{Special Case}.
If $|S| = 1$, then $|N^1_{\{u_1,u_2,u_3,u_4\}}(S)| = 0$. Without loss of generality, assume $S = \{v_1\}$, i.e., \[N^0_{\{u_1,u_2,u_3,u_4\}}(v_1) = \{u_1,u_2,u_3,u_4\}.\] 
If $\omega(v_2u_1) = 0$, then  $\omega(u_2v_3)=\omega(u_2v_4)=0$ according to the structure of the \textbf{Special Case}. This yields $\{v_1, v_3, v_4\}\subseteq N^0_{\{v_1,v_2,v_3,v_4\}}(u_2)$, a contradiction to (\ref{EQ:cl2}).  Therefore, $\omega(v_2u_1)=1$. By symmetry, we have  $\omega(v_iu_j) = 1$ for all $2\le i \le 4$ and $1\le j\le 4$. If one of $\{v_2w,v_3w,v_4w\}$ has weight zero, say $v_2w$, by adding \[v_1T_1(u_1,u_2,w)w,v_2w,v_3u_3w,v_4u_4w\]to $\sub(H')$, we obtain a 2-divisible subdivision of $H$.  Thus $\omega(v_2w)= \omega(v_3w) = \omega(v_4w) = 1$. Then by adding 
\[v_1u_1,v_2u_2u_1,v_3wu_1,v_4T_1(v_4,u_3,u_4)u_1\]
to $\sub(H')$, we again obtain a 2-divisible subdivision of $H$.

\vspace{5pt}    
 \noindent\textbf{(4.2)} $G$ contains no $F_1$ but contains an $F_2$ or an $F_3$. Let $F$ be an $F_2$ or an $F_3$ with $\omega(u_1u_2) = 1$. 
     
      If  the \textbf{Special Case} does not hold, then by symmetry, both $u_1$ and $u_2$ have more than two 1-neighbors in $\{v_1,v_2,v_3,v_4\}$. Consequently, they share at least 2 common 1-neighbors, say $v_1,v_2$. If either $\omega(v_3w)=0$ or $\omega(v_4w)=0$, say $\omega(v_3w)=0$,  then by adding the zero-weight  paths $wv_3, wu_1v_1, wu_2v_2$,  and   $wT(u_3,u_4,w)v_4$ to $\sub(H')$,  we obtain a 2-divisible subdivision of $H$. Now assume $\omega(v_3w) = \omega(v_4w) = 1$. Since $u_1$ must have another 1-neighbor in $\{v_3,v_4\}$, say $v_3$, we have an $F_1$ formed by two triangles sharing the vertex $u_1$:  $T_1(v_2, u_1, u_2)$ and $T(v_3, u_1, w)$, a contradiction. 
     
    Now assume the \textbf{Special Case} holds. Then $\omega(u_1v_3)=\omega(u_1v_4)=1$ and $\omega(u_2v_1)=\omega(u_2v_2)=1$. Since the edges $v_i u_2$ and $v_j u_1$ for $1 \le i \le 2$ and $3 \le j \le 4$ form a matching in which both edges have weight 1, we must assign $\omega(v_{i'} w) = 1$ for all $1 \le i' \le 4$ to avoid creating a 2-divisible subdivision of $H$ with $w$ as the branch vertex corresponding to $u$ (otherwise, a matching $\{v_i u_2, v_j u_1\}(i,j\ne i')$ and the zero-weight edge $wv_{i'}$, together with the edges $w u_1$, $w u_2$, and the triangle $T(u_3, u_4, w)$ where $T \in \{T_1, T_3\}$,  would result in a 2-divisible subdivision of $H$).
       If $F=F_2$  with $\omega(wu_3) = 1$, then we must have $\omega(u_3v_i) = 0$ for all $1 \le i \le 4$ to avoid an $F_1$ (otherwise, $T_1(u_1, u_2, w)$ and $T_1(w, u_3, v_i)$ form an $F_1$), which contradicts \eqref{EQ:cl2}. If $F=F_3$, we must have $\omega(v_1u_3) = \omega(v_1u_4) = 1$ to avoid  an $F_2$ (otherwise, $T_1(u_1,u_2,w)$ and $T_3(w,v_1,u_3(u_4))$ would form an $F_2$).    Consequently, $T_1(v_1,u_2,w)$ and $T_1(v_1,u_3,u_4)$  form an $F_1$ with shared vertex $v_1$, a contradiction.

    \vspace{5pt}
    \noindent\textbf{(4.3)} $G$ contains none of  $\{F_1, F_2, F_3\}$ but contains an $F_4$ or an $F_5$ with $\omega(u_1w) = 1$. 

Let $F$ be an $F_4$ or an $F_5$. Since $|N^1_{\{v_1,v_2, v_3,v_4\}}(u_1)|\ge 2$,  without loss of generality,   
assume $\omega(v_1u_1) = \omega(v_2u_1) = 1$.  To prevent the emergence of $T_1$-type triangles $T_1(w,u_1,v_1)$ and $T_1(w,u_1,v_2)$ that would imply an $F'\in \{F_1,F_2,F_3\}$ (a contradiction), we have $\omega(v_1w) = \omega(v_2w) = 0$. Therefore, by adding the edges $wv_1$, $wv_2$, and the two zero-weight paths $wT_3(u_1,u_2,w)v_3$ and $wT_3(u_3,u_4,w)v_4$, we obtain a 2-divisible subdivision of $H$. 

\vspace{5pt}    
\noindent\textbf{(4.4)} $G$ contains none of  $\{F_1, F_2, \dots, F_5\}$ but contains an $F_6$. 
     
Let $F$ be an $F_6$. 
 If $u_1$ and $u_2$ share a common 1-neighbor in $\{v_1,v_2,v_3,v_4\}$, say $v_1$, then we obtain a $T_1$-type triangle $T_1(u_1, u_2, v_1)$. Together with the $T_3$-type triangle $T_3(w,u_3,u_4)$, this gives two disjoint triangles--one being $T_1$-type and the other $T_3$-type. From the proof of Claim~\ref{CL:F}, it follows that there exists an $F'\in\{F_1, F_2, F_3\}$, a contradiction. Therefore, $u_1$ and $u_2$ have no common 1-neighbors in $\{v_1,v_2,$ $v_3,v_4\}$.  Since $|N^1_{\{v_1,v_2, v_3,v_4\}}(u_i)|\ge 2$ for $i=1,2$, we have $|N^1_{\{v_1,v_2, v_3,v_4\}}(u_1)|=|N^1_{\{v_1,v_2, v_3,v_4\}}(u_2)|=2$. Without loss of generality,   
  assume $\omega(v_1u_1) = \omega(v_2u_1) = \omega(v_3u_2) = \omega(v_4u_2) = 1$. This implies that $\omega(v_1u_2) = \omega(v_2u_2) = \omega(v_3u_1) = \omega(v_4u_1) = 0$. 
   If there exists a vertex $v_i$ satisfying $\omega(wv_i)=0$, say $\omega(wv_1)=0$, then by adding the zero-weight paths $wv_1$, $wu_2v_2$, $wu_1v_3$, and $wT_3(w, u_3, u_4)v_4$, we obtain a 2-divisible subdivision of $H$. Now assume  $\omega(v_iw) = 1$ for $1\le i\le 4$. Then $T_3(v_1,u_2,w)$ and $T_3(u_3,u_4,w)$ forms an $F_5$, a contradiction.

    \begin{figure}[ht]
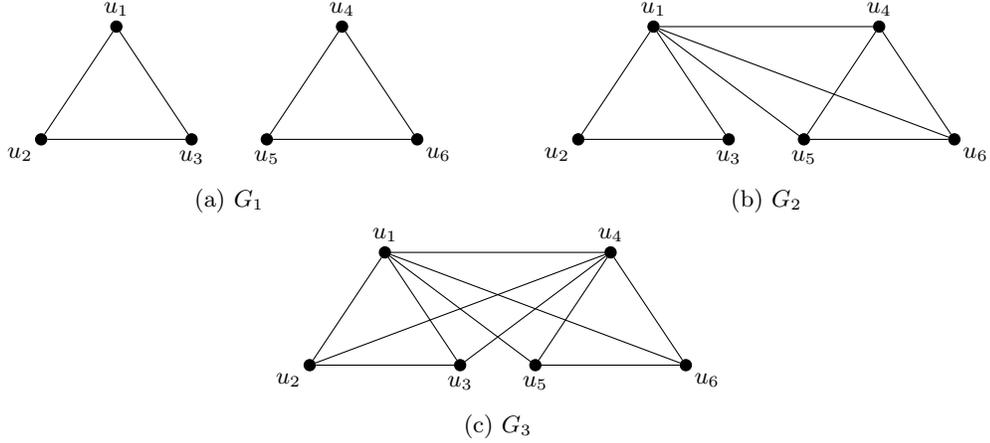

        \centering
        \subfloat[$G_1$]{\twotrif} \qquad \subfloat[$G_2$]{\twotris}  \quad \subfloat[$G_3$]{\twotrit} 
        \caption{The structure used in Case $\delta(H) = 5$.}\label{Fig:TwoT1}
    \end{figure}
\vspace{10pt}
    \noindent{\bf (5)} $\delta(H)=5$.
    
Let $u$ be  a vertex  of degree 5 in $H$ and denote $N_H(u)=\{v_1,v_2,v_3, v_4,v_5\}$. 
Since $\delta(H) =  5$, we have $n\ge 6$ and $n + m\ge n + 5n/2 = 7n/2$. If $G$ contains exactly one triangle of $T_1$-type or $T_3$-type, then deleting it results in a complete graph $K_{f'}$  with $f' = n + m -3\ge 7n/2 - 3> 2n - 1$ that contains no triangles of types $T_1$ and $T_3$. Therefore, according to Lemma~\ref{bipar}, it must contain a copy of $H$ with all edges of weight zero. Thereby, we can assume that $G$ contains  triangles $T(u_1,u_2,u_3)$ and $T(u_4,u_5,u_6)$ with $T\in \{T_1,T_3\}$. We may further assume they are disjoint; otherwise,  deleting one would destroy the other.
Let $U = \{u_i|1\le i\le 6\}$ and $V = \{v_i|1\le i\le 5\}$. By the induction hypothesis, $G - U$ contains a 2-divisible subdivision  $\sub(H')$ of $H'$.

   \begin{cla}\label{N2}
      $|N_V^0(u_i)|\le 2$, or equivalently, $|N_V^1(u_i)|\ge 3$ for $1\le i\le 6$. 
   \end{cla}
\begin{proof}
   If $|N_V^0(x)| \ge 3$ for some  $x\in U$, say $u_1$, let $v_1,v_2,v_3\in N^0_V(u_1)$, then by  adding the zero-weight  edges $u_1v_1, u_1v_2, u_1v_3$ and the zero-weight  paths $u_1T(u_1,u_2,u_3)v_4$ and $u_1T(u_4,u_5,u_6)v_5$ to $\sub(H')$, we  obtain a 2-divisible  subdivision of $H$. 
 
\end{proof}
     
    
 Let $G_1$, $G_2$, and $G_3$ be the structures, each consisting of two $T_1$-type triangles connected by edges of weight 1 (as shown in Figure~\ref{Fig:TwoT1}).
\begin{cla}\label{Fig:G1}
   There exists a $G_1$ in $G$. 
\end{cla}
    
\begin{proof}
        Assume, to the contrary, that one of the triangles—say, $T(u_1,u_2,u_3)$—is of type $T_3$. Suppose further that $\omega(u_2u_3) = 1$. By Claim \ref{N2}, we have $|N_V^1(u_i)| \ge 3$ for $i = 2, 3$. Since $|V| = 5$, the pigeonhole principle implies that $u_2$ and $u_3$ share a common 1-neighbor in $V$. This results in a new triangle of type $T_1$ in $G$, thereby allowing the replacement of $T_3(u_1,u_2,u_3)$ with it, which leads to a contradiction.
\end{proof}

Thus, we assume that $G[U]$ contains a $G_1$.

    \begin{cla}\label{CL:G2}
       There exists a $G_2$ in $G$. 
    \end{cla}

      \begin{proof}
     First, we assert that 
\begin{equation}\label{Eq:ui1}
    |N_V^0(u_i)|\le 1 \,\text{( or equivalently, $|N_V^1(u_i)|\ge 4$) for $1\le i\le 6$}.
\end{equation}
Suppose not. Without loss of generality, assume $|N_V^0(u_1)| =  2$, and denote $N^0_V(u_1)=\{v_1,v_2\}$. Since $|N_V^1(u_2)| \ge  3$,  $N_V^1(u_2)\cap\{v_3,v_4,v_5\}\not=\emptyset$. Without loss of generality, assume $\omega(u_2v_3) = 1$. 
If at least one of the edges $u_3v_4$, $u_3v_5$ has weight 1, say $\omega(u_3v_4)=1$, then by adding the zero-weight edges $u_1v_1$, $u_1v_2$, along with the 
 zero-weight paths $u_1u_2v_3$, $u_1u_3v_4$ and a zero-weight path $u_1T(u_4,u_5,u_6)v_5$ (guaranteed by Observation~\ref{spetri}) to $\sub(H')$, we obtain a 2-divisible subdivision of $H$. Hence, it must be that $\omega(u_3v_4) = \omega(u_3v_5) = 0$, which implies that $\omega(u_3v_i)=1$ for $1\le i\le 3$. 
 If one of $v_1u_2$, $v_2u_2$ has weight 1, say $\omega(v_1u_2)=1$, then by adding the zero-weight edges $u_3v_4$, $u_3v_5$, along with the zero-weight paths $u_3u_2v_1$, $u_3u_1v_3$, and a zero-weight path $u_3T_1(u_4,u_5,u_6)v_2$ (guaranteed by Observation~\ref{spetri})  to $\sub(H')$, we again obtain a 2-divisible subdivision of $H$. Therefore, $\omega(v_1u_2) = \omega(v_2u_2) = 0$. This yields that $\omega(u_2v_i)=1$ for $i=3,4,5$ by Claim~\ref{N2}. Then by adding the zero-weight edges $u_1v_1$, $u_1v_2$, together with the zero-weight paths $u_1u_3v_3$, $u_1u_2v_4$, and a zero-weight path $ u_1T_1(u_4,u_5,u_6)v_5$ (guaranteed by Observation~\ref{spetri}) to $\sub(H')$, we once more obtain a 2-divisible subdivision of $H$. 
 
Now consider the number of edges of weight 1 connecting a subset $\{u_i,u_j,u_k\}\subseteq U$  to $V$.  By (\ref{Eq:ui1}), we have 
\begin{equation}\label{EQ:vi1}
12\le |E^1(\{u_i,u_j,u_k\}, V)|=\sum_{\ell=1}^5|N_{\{u_i,u_j,u_k\}}^1(v_\ell)|.
\end{equation}
It follows from (\ref{EQ:vi1}),  there are at least two vertices, say $v_1,v_2$, in $V$ satisfying $|N_{\{u_1,u_2,u_3\}}^1(v_i)|=3$, i.e., $N_{\{u_1,u_2,u_3\}}^1(v_i)=\{u_1,u_2,u_3\}$ for $i=1,2$. 
Furthermore, by (\ref{EQ:vi1}), there is at most one vertex $v$ in $V$ satisfying $|N_{\{u_i,u_j,u_k\}}^1(v)|\le 1$. Hence  at least one of $v_1,v_2$ satisfies $|N_{\{u_4,u_5,u_6\}}^1(v_i)|\ge 2$. Without loss of generality, assume $\omega(v_1u_4) = \omega(v_1u_5) = 1$. Then the triangles $T_1(u_1,u_2,u_3)$ and $T_1(v_1,u_4,u_5)$ form the desired structure $G_2$.

\end{proof}
    \begin{cla}\label{G3}
        There exists a $G_3$ in $G$.
    \end{cla}
    
\begin{proof}
 It follows from Claim~\ref{CL:G2} that we may assume that $G[U]$ contains a $G_2$. We assume that there is no perfect matching between $U\setminus\{u_1\}$ and $V$ using only edges of weight 1; otherwise, we could add this matching and the edges $\{u_1u_i | 2\le i\le 6\}$ to $\sub(H')$, thereby obtaining a 2-divisible subdivision of $H$.  Since $|N_V^1(u)| \ge  4$ for any $u\in U$, by Hall's Theorem, the only obstruction to obtain a perfect matching between $U\setminus\{u_1\}$ and $V$ using only edges of weight 1 occurs when $|N_V^1(u)|=4$ and the sets $N_V^1(u)$ are identical for all $u\in U\setminus\{u_1\}$, that is, there exists a vertex, say $v_1$, in $V$ with $N^0_{U\setminus\{u_1\}}(v_1)=U\setminus\{u_1\}$. 
 Hence $\omega(v_iu_j) = 1$ for all $2\le i\le 5$ and $2\le j\le 6$. If $\omega(u_1v_2) = 0$, then by adding the edge $u_1v_2$ and the zero-weight paths $u_1T_1(u_1,u_2,u_3)v_1, u_1u_4v_3, u_1u_5v_4, u_1u_6v_5$ to $\sub(H')$, we obtain a 2-divisible subdivision of $H$. Now assume $\omega(u_1v_2)=1$. Then the triangles $T_1(u_1,u_2,u_3)$ and $T_1(v_2,u_4,u_5)$, together with the edges of weight 1 between them,  form a $G_3$ on $G$.
\end{proof}
  According to Claim~\ref{G3},  we may assume that $G[U]$ contains a $G_3$.  By symmetry  of $u_1$ and $u_4$ in $G_3$, and the proof of Claim~\ref{G3}, we must have $N^0_{U\setminus\{u_1\}}(v_1)=U\setminus\{u_1\}$ and $N^0_{U\setminus\{u_4\}}(v_1)=U\setminus\{u_4\}$, i.e., $v_1$ is the only 0-neighbor of all $u_i$, $1\le i\le 6$. Then by adding the edge $u_1v_1$ and the zero-weight paths $u_1u_2v_2, u_1u_3v_3, u_1u_4v_4, u_1u_5v_5$ to $\sub(H')$, we obtain a 2-divisible subdivision of $H$. 
\end{proof}

\section{Proof of Theorem~\ref{ts}}
We restate Theorem~\ref{ts} here.
\onesubTree*

\subsection{Proof of Theorem~\ref{ts} (1)}

It suffices to prove that every $\mathbb{Z}_2$-edge-weighted $K_{2n - 1}$ contains a  2-divisible 1-subdivision of $T$. 

The proof is by induction on $|V(T)|$. The result holds trivially for the base cases $n = 1$ and $n = 2$. Assume $n\ge 3$ and the result holds for all trees on fewer than $n$ vertices. 
Let $T$ be a tree on $n$ vertices. 
By contradiction, assume there is a $\mathbb{Z}_2$-edge-weighted $K_{2n - 1}$ that does not contain a 2-divisible 1-subdivision of $T$. This yields that this edge weight function $\omega: E(K_{2n-1})\to \mathbb{Z}_2$ could not be monochromatic. 
\begin{cla}\label{twocolor}
    For every edge $x_1x_2$, there exists another vertex $x_3$ such that $\omega(x_1x_3) =\omega(x_2x_3)$ but different from $\omega(x_1x_2)$. 
\end{cla}
\begin{proof}
Otherwise, suppose the edge $x_1x_2$ does not apply to the Claim. According to the induction hypothesis, $K_{2n-1} - \{x_1, x_2\}$ contains a 2-divisible 1-subdivision $\sub_1(T - v)$ of $T - v$, where $v$ is a leaf in $T$. Let $u$ be the neighbor of $v$ in $T$.
By assumption, either $ux_1$ or $ux_2$ must have the same weight as $x_1x_2$. By adding this edge and $x_1x_2$ to $\sub_1(T - v)$, we obtain a 2-divisible 1-subdivision of $T$, a contradiction. 
\end{proof}

Note that in any tree, there is either a single leaf attached to a degree-two vertex or two leaves attached to a common vertex. We distinguish the proof according to the two structures.

\begin{figure}[htbp]
    \centering
    \subfloat[Case 1.]{\begin{tikzpicture}\label{fig:dtu}
    \begin{scope}[every node/.style={circle, fill, draw,inner sep = 1.5pt}]
        \path (0,0) node(a)  {} 
          (2,0) node(b) {}
          (1,1.5) node(c) {}
          (3,1.5) node(d) {};
    \end{scope}
    \begin{scope}
    \draw (node cs:name=a)node [below] {$x_1$} -- (node cs:name=b) node[midway,below] {$0$} node[below] {$x_2$} -- (node cs:name=c) node[above right,midway] {$1$};
    \draw (node cs:name=c) node[above] {$x_3$}  -- (node cs:name=a)node[above left,midway] {$1$};
    \draw (node cs:name=d) node[above] {$x_4$} --  (node cs:name=b)node[below right,midway] {$0$};
    \draw (node cs:name = c) -- (node cs:name = d) node[above,midway] {$0$};
    \end{scope}
   \end{tikzpicture}}
    \hspace{1.2in} \subfloat[Case 2.]{\label{fig:fourvertexmat}
    \centering
    \begin{tikzpicture}
    \begin{scope}[every node/.style={circle, fill, draw,inner sep = 1.5pt}]
        \path (0,0) node(a)  {} 
          (1.5,0) node(b) {}
          (1.5,1.5) node(c) {}
          (0,1.5) node(d) {};
    \end{scope}
    \begin{scope}
    \draw (node cs:name=a)node [below] {$x_1$} -- (node cs:name=b) node[midway,above] {$1$} node[below] {$x_2$} -- (node cs:name=c) node[right,midway] {$0$};
    \draw (node cs:name=d) node[above] {$x_4$} --  (node cs:name=a)node[left,midway] {$0$};
    \draw (node cs:name = c) node[above] {$x_3$} -- (node cs:name = d) node[above,midway] {$1$};
    \end{scope}
   \end{tikzpicture}
}
\caption{The structure in the proof of Theorem \ref{ts} (1).}\label{tree:fourgraph}
\end{figure}
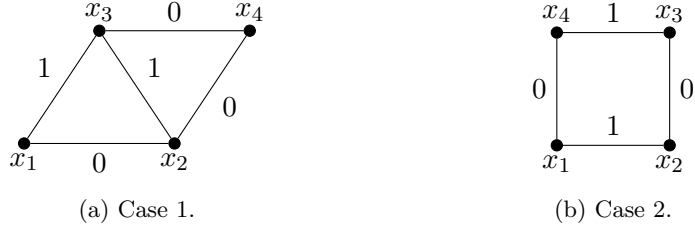

\noindent\textbf{Case 1.} There exists a leaf $v$ adjacent to a degree-two vertex $u$.  

Assume that $N_T(u)=\{v,w\}$. According to Claim \ref{twocolor}, there are four vertices that produce the graph depicted in Figure \ref{tree:fourgraph}(a) in $K_{2n - 1}$. By the induction hypothesis, there exists a 2-divisible $\sub_1(T-\{u,v\})$ in $K_{2n - 1} - \{x_1,x_2,x_3,x_4\}$. If $\omega(wx_1) = 0$, by adding the subdivision paths $x_2x_1w$ and $x_3x_4x_2$ to $\sub_1(T-\{u,v\})$, we obtain a 2-divisible 1-subdivision of $T$, a contradiction. If $\omega(wx_1) =1$, by adding the subdivision paths $x_3x_1w$ and $x_2x_4x_3$ to $\sub_1(T-\{u,v\})$, we again obtain a 2-divisible 1-subdivision of $T$, a contradiction.

\noindent\textbf{Case 2.} There exist two leaves $u,v$ attached to the common vertex $w$. 

By the induction hypothesis, there exists a 2-divisible 1-subdivision $\sub_1(T - \{u,v\})$.  
\begin{cla}\label{fourcycle}
 $K_{2n - 1}$ contain no set of four vertices $\{x_1, x_2, x_3, x_4\}$  that forms the subgraph shown in Figure~\ref{tree:fourgraph} (b). 
\end{cla}

\begin{proof}
    
 Suppose $K_{2n-1}$ contains a subgraph depicted in Figure~\ref{tree:fourgraph} (b). Assume, without loss of generality, that $\omega(wx_1) = 0$. If $\omega(wx_3) = 0$ or $\omega(wx_2) = 0$, then by adding paths $wx_1x_4$, and  $wx_3x_2$ or $wx_2x_3$ to $\sub_1(T-\{u,v\})$, we obtain a 2-divisible 1-subdivision of $T$, a contradiction. Hence $\omega(wx_3) = \omega(wx_2) = 1$. Then by adding the paths $wx_3x_4$ and $wx_2x_1$ to $\sub_1(T-\{u,v\})$, we obtain a 2-divisible 1-subdivision of $T$, a contradiction again.

\end{proof}

\begin{cla}\label{CL:order}
There exists an ordering  $x_1,x_2,\dotsc, x_{2n - 1}$ of  the  vertices in $K_{2n - 1}$ such that for all $1\le i < j\le 2n - 1$,
    \[
    \omega(x_ix_j) = \begin{cases}
        0, &\text{if $j$ is odd};\\
        1, &\text{if $j$ is even}.
    \end{cases}
    \]
\end{cla}

\begin{proof}
One can find $x_1$, $x_2$, and $x_3$ using Claim \ref{twocolor}. Assume that $x_1,\dotsc,x_t(t\ge 3)$ have been ordered. We proceed to  find $x_{t + 1}$.  Without loss of generality, assume $t$ is even, then $\omega(x_ix_t) = 1$ for all $1\le i < t$.  By applying Claim~\ref{twocolor} to the edge $x_{t-1}x_t$, there must exist another vertex $y$ satisfying  $\omega(x_ty) = \omega(x_{t - 1}y) = 0$.
Moreover, $\omega(x_ty)= 0$ implies that $y\notin\{x_1, x_2, \dots, x_{t-2}\}$. 
 By applying Claim \ref{fourcycle} to the 4-cycle $yx_tx_{t - 1}x_{i}y$ with $i< t - 1$, we must have $\omega(x_iy) = 0$. 
Therefore, $y$ can be arranged to $x_{t + 1}$.   
\end{proof}

Based on Claim~\ref{CL:order}, all edges incident to $x_{2n - 1}$ have color $0$, which contradicts with Claim~\ref{twocolor}.

\subsection{Proof of Theorem~\ref{ts} (2)}

We prove a stronger result. 

\begin{thm}\label{thm:cycle}
    Let $n$ be an integer with $n\ge 3$. Every red-blue edge-coloring of $K_{2n}$ contains a Hamiltonian cycle consisting of a red path and a blue path,  both of even length.  
\end{thm} 
Clearly, Theorem~\ref{ts} (2) is a direct corollary of Theorem~\ref{thm:cycle}.
The following theorem is useful for our proof. 

\begin{thm}[\cite{BT}]\label{partition}
         Every complete graph $K$ of order $n\ge 5$,  whose edges are colored with red and blue, admits a vertex set partition $ V_1\cup V_2$ with $|V_1|\ge |V_2|$ such that one of the following holds:

      (1) $|V_1|\ge |V_2|\ge 3$, and each of $K[V_1]$ and $K[V_2]$ contains a monochromatic Hamiltonian cycle,  with the two cycles being of different colors.

      (2) $|V_2| = 2$, and $K[V_1]$ has a monochromatic Hamiltonian cycle whose color is different from that of the edge in $K[V_2]$.

      (3) $K$ itself has a monochromatic cycle of length $n$ or $n - 1$.
\end{thm}

\begin{proof}[Proof of Theorem \ref{thm:cycle}]
Suppose, for contradiction, that there exists a red-blue edge-coloring $\omega: E(K_{2n}) \to \{r, b\}$ of $K_{2n}$ that admits no  Hamiltonian cycle of the desired type. By Theorem \ref{partition}, we have a partition of $V(K_{2n})$ into $V_1\cup V_2$ such that (1), (2), or (3) holds

    \vspace{5pt}
    \noindent\textbf{Case 1}.  The (3) holds, i.e., $K_{2n}$ contain a monochromatic cycle of length $2n$ or $2n - 1$.
    
    When the length of the monochromatic cycle is $2n$, we have a desired Hamilton cycle, a contradiction.  Now suppose the monochromatic cycle is $ x_1x_2\dotsc x_{2n - 1}x_1$, and the remaining vertex is $y$. Since the length of the monochromatic cycle is odd, there must exist an $i$ such that $x_iy$ and $x_{i + 1}y$ have the same color. 
    Then we obtain a Hamiltonian cycle consisting of monochromatic paths  $x_{i + 1}x_{i  + 2}\dots x_{2n-1}x_{1}x_2\dots x_{i}$ and  $x_iyx_{i + 1}$, both of length even, a contradiction again.

  \vspace{5pt}  
    \noindent\textbf{Case 2}. The (2) holds, i.e., $|V_2| = 2$, and $K_{2n}[V_1]$ has a monochromatic Hamiltonian cycle whose color is different from that of the edge in $K[V_2]$.

    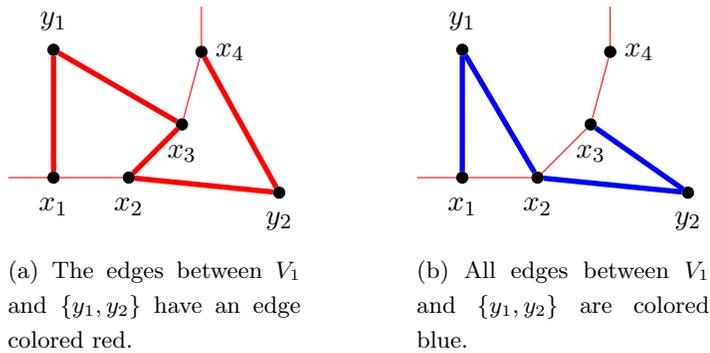
\begin{figure}[htbp!]
        \centering
        \subfloat[The edges between $V_1$ and $\{y_1,y_2\}$ have an edge colored red.]{\begin{tikzpicture}
            \begin{scope}[every node/.style={circle, fill, draw,inner sep = 1.5pt}]
            \path (0,0) node[label = below:$x_1$] (a) {} -- (1,0) node[label = below:$x_{2}$] (b) {} -- ++(45:1) node[label = below:$x_{3}$] (c) {} -- ++(75:1) node[label = right:$x_{4}$] (d) {};
            \path (0,1.7) node[label = above: $y_1$] (e){};
            \path (3,-0.2) node[label = below:$y_2$] (f){};
            \end{scope}
            \begin{scope}[red] 
            \draw (-0.6,0) --  (a) -- (b) -- (c) -- (d) -- +(90:0.6);
            \begin{scope}[line width = 2pt]
            \draw (b) -- (c);
            \draw (e) -- (a);
            \draw (e) -- (c);
            \draw (f) -- (b);
            \draw (f) -- (d);
             \end{scope}
            \end{scope}
        \end{tikzpicture}}
        \qquad \qquad
        \subfloat[All edges between $V_1$ and $\{y_1,y_2\}$ are  colored blue.]{\begin{tikzpicture}
            \begin{scope}[every node/.style={circle, fill, draw,inner sep = 1.5pt}]
            \path (0,0) node[label = below:$x_1$] (a) {} -- (1,0) node[label = below:$x_{2}$] (b) {} -- ++(45:1) node[label = below:$x_{3}$] (c) {} -- ++(75:1) node[label = right:$x_{4}$] (d) {};
            \path (0,1.7) node[label = above: $y_1$] (e){};
            \path (3,-0.2) node[label = below:$y_2$] (f){};
            \end{scope}
            \begin{scope}[red] 
            \draw (-0.6,0) --  (a) -- (b) -- (c) -- (d) -- +(90:0.6);
            \begin{scope}[blue,line width = 2pt]
            \draw (e) -- (a);
            \draw (e) -- (b);
            \draw (f) -- (b);
            \draw (f) -- (c);
             \end{scope}
            \end{scope}
        \end{tikzpicture}}
        \caption{The structure of Case 2.}
    \end{figure}

   Let $C_r = x_1x_2\dotsm x_{2n - 2}x_1$ be a Hamiltonian cycle in $K_{2n}[V_1]$, which is colored  red. Let $V_2 = \{y_1,y_2\}$ with $y_1y_2$ colored blue. Then $|C_r| \ge 4$ since $n \ge 3$. If there exists an edge $y_jx_i$ that is colored red, say $y_1x_1$, then $y_2x_2$ is colored red; otherwise, $y_1y_2x_2$ and $x_2x_3\dotsm x_{2n - 2}x_1y_1$ form a desired Hamiltonian cycle, a contradiction. 
   Similarly, $y_1x_3$ and $y_2x_4$ are colored red. Therefore, $x_1y_1x_3x_2y_2x_4x_5\dotsm x_{2n - 2}x_1$ is a monochromatic Hamiltonian cycle, a contradiction too. Thus, all edges between $C_r$ and $y_1,y_2$ are colored blue. Then $x_1y_1x_2y_2x_3$ and $x_3 x_4\dotsm x_{2n - 2}x_1$ once more form a desired Hamiltonian cycle, a contradiction. 
    
\vspace{5pt}
\noindent\textbf{Case 3}. The (1) holds, i.e.,  $|V_1|\ge |V_2|\ge 3$, and each of $K_{2n}[V_1]$ and $K_{2n}[V_2]$ contains a monochromatic Hamiltonian cycle,  with the two cycles being of different colors. 

Let $C_r$ (red) and $C_b$ (blue) be two monochromatic Hamiltonian cycles in $K_{2n}[V_1]$ and $K_{2n}[V_2]$,  respectively. 
    
If  $|C_b|\ge 3$ is odd, then $|C_r|$ is odd too. Let $C_r = x_1x_2\dots x_{2s - 1}$ and $C_b = y_1y_2\dots y_{2t - 1}$. If   there exist $i,j$ with $1\le i\le 2s-1$  and $1\le j\le 2t-1$ such that the edges $x_iy_{j}$ and $x_{i + 1}y_{j + 1}$ have the  same color, say color $r$, then $(C_r-\{x_ix_{i+1}\})\cup\{x_iy_j, x_{i+1}y_{j+1}\}$ and $C_b-\{y_{j}y_{j+1}\}$ form a desired Hamiltonian cycle, a contradiction. Thus $x_iy_j$ and $x_{i + 1}y_{j + 1}$ have different colors for all $1\le i\le 2s - 1$ and $1\le j\le 2t - 1$.  Assume $x_1y_1$  is colored  $r$. 
Then $x_{1+i}y_{1+i}$ is colored  $r$ if $i$ is even and  $b$ if $i$ is odd, where  the indices of $x$ and $y$ are taken  modulo $2s-1$ and $2t-1$, respectively.  Consequently, $x_{1 + (2s - 1)(2t - 1)}y_{1 + (2s - 1)(2t - 1)}=x_1y_1$ is colored $b$, leading to a contradiction.

Now assume $|C_b|\ge 4$ is even. Let $C_r = x_1x_2\dots x_{2s}$ and $C_b = y_1y_2\dots y_{2t}$. By assumption, $s\ge t$. 
If all edges between $C_r$ and $C_b$ are colored $b$, then $x_1y_1x_2y_2\dotsm x_{2t}y_{2t}x_{2t + 1}$ and $x_{2t + 1}x_{2t + 2}\dotsm x_{2s}x_1$ constitute  a desired cycle ($x_{2t + 1}$ may be $x_1$), a contradiction. 
Thus, there exists an edge  $x_{k}y_{l}$ is colored $r$ for some $1\le k\le 2s$ and $1\le l\le 2t$. We assert that for every $x\in\{x_{k - 1},x_{k + 1}\}$ and $y\in\{y_{l - 1},y_{l + 1}\}$, $xy$ are colored $r$; otherwise, $(C_r\cup C_b - \{xx_k,yy_l\})\cup \{xy,x_ky_l\}$ constitute a desired cycle, a contradiction. Without loss of generality, assume $x_1y_1$ is colored $r$, then $x_2y_2$ is colored $r$. Consequently, $x_1y_3$ is colored $r$. This leads to $x_2y_4$ and $x_1y_5$ being colored $r$. Continue this process, we obtain that $x_1y_j$ is colored $r$ for odd $j$. From $x_2y_2$ is colored $r$, with a similar discussion, we can obtain that $x_2y_j$ is colored $r$ for even $j$. Continue the above process, we can conclude that $x_iy_j$ is colored $r$ for $i\equiv j \pmod{2}$. Therefore, \[x_1y_1x_3y_3\dotsm x_{2t - 3}y_{2t - 3}x_{2t - 1}x_{2t - 2}y_{2t - 2}x_{2}y_{2}\dotsm x_{2t - 4}y_{2t - 4}x_{2t}x_{2t + 1}\dotsm x_{2s}x_1\] is a monochromatic cycle (red) of length $2n-2$. The rest edge $y_{2t - 1}y_{2t}$ has color $b$, we are back in \textbf{Case}   2.

 The proof is completed.
\end{proof}

\noindent \textbf{Remark.} The lower bound $n\ge 3$ is tight. For example, the red-blue colored  $K_4$ as shown in Figure~\ref{Fig:K4} does not have a Hamiltonian cycle consisting of a red path and a blue path, both of even length. 
\begin{figure}[htbp!]
    \centering
    \begin{tikzpicture}
    \begin{scope}[every node/.style={circle, fill, draw,inner sep = 1.5pt}]
        \path (0,0) node (a) {}-- (1.5,0) node (b){} -- (1.5,1.5) node (c) {}-- (0,1.5) node (d) {};
    \end{scope}
             \draw[red] (a) -- (b) -- (c) -- (d);
             \draw[blue] (b) -- (d) -- (a) -- (c);
    \end{tikzpicture}
    \captionsetup{justification = centering}
    \caption{A red-blue colored $K_4$ does not have a $C_4$ that consists of a \\ red path and a blue path, each of length 2 or monochromatic.}\label{Fig:K4}
\end{figure}
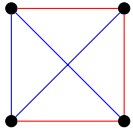

\section{Concluding and  Remarks}

In this paper, we investigate the issue of finding a specific subdivision with prescribed modular constraints on subdivision path lengths of a given graph among the subdivisions of $K_f$. Alon \cite{NogaAlon} proved that the Ramsey number of every subdivided graph with order $n$ is at most $12n$. Based on Alon's result,  for an integer $q\ge 2$ and a given graph $H$ with $n$ vertices and $m$ edges, we have  \[m(q - 1) + n \le s_q(H)\le s_q(H,q-1)\le 12(m(q - 1) + n).\]
When $H$ is a subcubic graph on $n$ vertices with $m$ edges. Das, Dragani\'{c} and Steiner \cite{SNR} proved that for any $f \ge 7qm + 8n +14q$, every $\mathbb{Z}_q$-edge-weighted $K_f$-minor contains a $q$-divisible $H$-subdivision. 
Based on this result, we also have $s_q(H) \le 7qm + 8n + 14q$. In this paper, we demonstrate that the term "subcubic" in this problem can actually be omitted under the subdivision condition. For any graph $H$ with $n$ vertices and $m$ edges, we prove that, for a prime $p\ge 3$ and a connected graph $H$, we have $$s_p(H) \le \frac{3p - 1}{2}m - \frac{p - 1}{2}n + \frac{p + 1}{2}.$$ For general $q$ and $H$, we have $$s_q(H)\le (2q - 1)m + 2n  - 1 + 4q.$$ 
Furthermore, we have also established several cases where the bounds are tight. For general $q$ and a tree $T$ with $n$ vertices, $$s_q(T) = nq - q + 1.$$
For $q = 2$, we complete solve Problem~\ref{Prob:DDS} by showing that $$s_2(H) = m + n$$ for every $5$-degenerate graph $H$ with $n$ vertices and $m$ edges. 
Moreover, we prove that $s_2(T,1) = s_2(C,1) = m + n$ for every tree $T$ and cycle $C$ on $n$ vertices with $m$ edges. 
Based on these results, it is natural to ask the following question.

\begin{prob}\label{PB: generalH}
    1. For every graph $H$ and any integer $q\ge 2$,  $s_q(H) = m(q  - 1) + n$.

    2. Whether $s_2(H,1) = m + n$ for every graph $H$? 
\end{prob}

Let $H$ be a connected graph. If we select any $\lfloor\frac{q - 1}{2}\rfloor$ vertices in $K_f$ and assign the value 1 to all edges incident to these vertices, and assign the value 0 to all the remaining edges, it is straightforward to verify that this configuration does not contain any $q$-divisible $(q - 1)$-subdivision of $H$ when $f = m(q - 1) + n + \lfloor\frac{q - 1}{2}\rfloor - 1$. Thus, $s_q(H,q - 1)\ge m(q - 1) + n  + \lfloor\frac{q - 1}{2}\rfloor$.  Currently, it appears that determining the exact value of $s_q(H, q - 1)$ for $q \ge 3$ remains quite challenging. Determining its range is likely to be an interesting problem.

When considering a family of graphs, for example, cycles of lengths divisible by $q$,  Alon and Krivelevich \cite{AlonK} proved that for some constant $C$, every $K_f$-minor with $f\ge C\cdot q\log q$ contains a cycle of length divisible by $q$. Later,  M\'{e}sz\'{a}ros and Steiner \cite{Meszaro} showed that $f\ge 16q$ is sufficient. Subsequently, the coefficient `16'  was improved to '4' by Berendsohn, Boyadzhiyska, and Kozma  \cite{Berendsohn} as well as by Akrami, Alon, Chaudhury, Garg,     Mehlhorn, and Mehta \cite{Akrami}. Recently, Campbell, Gollin, Hendrey, and Steiner \cite{Campbell} proved that every $K_{2q + 2}$-minor contains a cycle of length divisible by $q$ when $q$ is odd. For $q$ is even, they said there were additional obstacles to overcome and would be present in the following series of their articles. 
It is very interesting to investigate such a problem in  $K_f$-subdivisions(or the problem of finding a zero-weight cycle in  $\mathbb{Z}_q$-edge-weighted complete graphs).

\vspace{5pt}
\noindent{\bf Acknowledgements}:
This work was supported by the National Key Research and Development Program of China (2023YFA1010203), the National Natural Science Foundation of China (No.12471336), and the Innovation Program for Quantum Science and Technology (2021ZD0302902).


\begin{thebibliography}{99}
\bibitem{Akrami}  H. Akrami, N. Alon, B. R. Chaudhury, J. Garg, K. Mehlhorn, R. Mehta, EFX Allocations: Simplifications and Improvements, preprint, arXiv:2205.07638, 2022.
\bibitem{NogaAlon} N. Alon,  Subdivided graphs have linear ramsey numbers, J. Graph Theory 18 (1994) 343-347. 
\bibitem{AlonK} N. Alon, M. Krivelevich, Divisible subdivisions, J. Graph Theory 98 (2021) 623–629.
\bibitem{AlonKS} N. Alon, M. Krivelevich, B. Sudakov, Turán numbers of bipartite graphs and related Ramsey-type questions, Combin. Probab. Comput. 12 (2003) 477–494.
\bibitem{Berendsohn}   B. A. Berendsohn, S. Boyadzhiyska, L. Kozma, Fixed-point cycles and approximate EFX allocations, 
in: Proceedings of 47th International Symposium on Mathematical Foundations of Computer Science 
(MFCS 2022), LIPIcs, vol. 241, 2022, pp. 17:1–17:13.

\bibitem{BT} S. Bessy, S. Thomass\'{e}, Partitioning a graph into a cycle and an anticycle, a proof of Lehel’s conjecture, J. Combin. Theory, Ser. B 100(2) (2010), 176–180.


\bibitem{bollobas} B. Bollobás, A. Thomason, Proof of a conjecture of Mader, Erdős and Hajnal on topological complete
subgraphs, European J.
Combin. 19 (1998) 883–887.
\bibitem{Campbell} R. Campbell, J. P. Gollin, K. Hendrey, R. Steiner, Optimal bounds for zero-sum cycles. I. odd order, J. Combin. Theory, Ser. B 173 (2025) 246-256. 

\bibitem{SNR} S. Das, N. Dragani\'{c}, R. Steiner, Tight bounds for divisible subdivisions, J. Combin. Theory, Ser. B
165 (2024) 1–19.

 \bibitem{Daven} H. Davenport, On the addition of residue classes, J. London Math. Soc. 10 (1935) 30-32.

\bibitem{IJHOZ} I. G. Fern\'{a}ndez, J. Hyde, H. Liu, O. Pikhurko, Z. Wu, Disjoint isomorphic balanced clique subdivisions, J. Combin. Theory, Ser. B 161 (2023) 417-436. 

\bibitem{Frudi} Z. F\"{u}redi, D. J. Kleitman, On zero-trees, J. Graph Theory
 16(2) (1992) 107–120.

\bibitem{libinlong} E. Gy\H{o}ri, B. Li, N. Salia, C. Tompkins, K. Varga, M. Zhu, On graphs without cycles of length 0 modulo 4, J. Combin. Theory, Ser. B 176 (2026) 7-29.  

\bibitem{komlos} J. Koml\'{o}s, E. Szemer\'{e}di, Topological cliques in graphs II,  Combin. Probab. Comput. 5 (1996) 79–90.

\bibitem{LiuM} H. Liu, R. Montgomery, A solution to Erdős and Hajnal’s odd cycle problem, J. Amer. Math. Soc.
36 (2023) 1191–1234.
\bibitem{Meszaro}  T. M\'{e}sz\'{a}ros, R. Steiner, Zero sum cycles in complete digraphs, European J.
Combin. 98 (2021) 103399.
 \bibitem{Thomassen} C. Thomassen, Graph decomposition with applications to subdivisions and path systems modulo 
$k$, J. Graph Theory 7 (1983) 261–271.
\bibitem{Thomassen1}  C. Thomassen, Subdivisions of graphs with large minimum degree, J. Graph Theory
 8(1) (1984) 23-28.




\end{thebibliography}
\end{document}